\documentclass[reqno]{amsart}
\usepackage{amsmath,amsthm,amscd,amssymb,amsfonts, amsbsy}
\usepackage{latexsym, color, enumerate}
\usepackage{pxfonts}

\theoremstyle{plain}
\newtheorem{theorem}{Theorem}[section]
\newtheorem{lemma}[theorem]{Lemma}
\newtheorem{corollary}[theorem]{Corollary}
\newtheorem{proposition}[theorem]{Proposition}

\theoremstyle{definition}
\newtheorem{definition}[theorem]{Definition}

\newtheorem{example}[theorem]{Example}
\newtheorem{assumption}[theorem]{Assumption}

\theoremstyle{remark}

\newcommand{\dv}{\operatorname{div}}

\newcommand{\supp}{\operatorname{supp}}
\newcommand{\dist}{\operatorname{dist}}
\newcommand{\diam}{\operatorname{diam}}

\numberwithin{equation}{section}

\newcommand{\bI}{\mathbb{I}}

\newcommand{\bP}{\mathbb{P}}

\newcommand{\bR}{\mathbb{R}}

\newcommand\cA{\mathcal{A}}
\newcommand\cB{\mathcal{B}}

\newcommand\cD{\mathcal{D}}

\newcommand\cH{\mathcal{H}}

\newcommand\cN{\mathcal{N}}

\newcommand\cM{\mathcal{M}}

\providecommand{\set}[1]{\{#1\}}

\providecommand{\bigset}[1]{\bigl\{#1\bigr\}}

\providecommand{\abs}[1]{\lvert#1\rvert}

\providecommand{\norm}[1]{\lVert#1\rVert}

\renewcommand{\vec}[1]{\boldsymbol{#1}}

\def\Xint#1{\mathchoice
	{\XXint\displaystyle\textstyle{#1}}%
	{\XXint\textstyle\scriptstyle{#1}}%
	{\XXint\scriptstyle\scriptscriptstyle{#1}}%
	{\XXint\scriptscriptstyle\scriptscriptstyle{#1}}%
	\!\int}
\def\XXint#1#2#3{{\setbox0=\hbox{$#1{#2#3}{\int}$}
		\vcenter{\hbox{$#2#3$}}\kern-.5\wd0}}
\def\dashint{\Xint-}

\newcommand{\p}{\partial}
\newcommand{\epsi}{\varepsilon}

\begin{document}



	\title[mixed boundary value problem]{The Dirichlet-conormal problem with homogeneous and inhomogeneous boundary conditions}
	
%
%

	\author[Hongjie Dong]{Hongjie Dong}	
	
	\address{
Division of Applied Mathematics, Brown University, 182 George Street, Providence, RI 02912, USA}
	
	\email{hongjie\_dong@brown.edu}
\thanks{H. Dong and Z. Li were partially supported by the NSF under agreement DMS-1600593.}

	\author[Zongyuan Li]{Zongyuan Li}
	
	\address{Division of Applied Mathematics, Brown University, 182 George Street, Providence, RI 02912, USA}
	
	\email{zongyuan\_li@brown.edu}
	
\begin{abstract}
We consider the mixed Dirichlet-conormal problem on irregular domains in $\bR^d$. Two types of regularity results will be discussed: the $W^{1,p}$ regularity and a non-tangential maximal function estimate. The domain is assumed to be Reifenberg-flat, and the interfacial boundary is either Reifenberg-flat of co-dimension $2$ or is locally sufficiently close to a Lipschitz function of $m$ variables, where $m=1,\ldots,d-2$. For the non-tangential maximal function estimate, we also require the domain to be Lipschitz.
\end{abstract}
\maketitle

\section{Introduction}
In this paper, we continue our discussion in \cite{CDL} on the mixed Dirichlet-conormal boundary value problems. On a domain $\Omega\subset \bR^d$, we consider the following second-order symmetric divergence form elliptic equation, with two types of boundary conditions prescribed on two different parts of the boundary:
\begin{equation}		\label{eqn-11291710}
\begin{cases}
L u = f+ D_if_i  & \text{in }\, \Omega,\\
Bu = f_i  n_i + g_\cN& \text{on }\, \cN,\\
u = g_\cD & \text{on }\, \cD.
\end{cases}
\end{equation}
Here, the boundary is decomposed into two non-empty and non-intersection portions $\cN$ and $\cD$, separated by their interfacial boundary $\Gamma$:
\begin{equation*}
\p\Omega = \cN\cup\cD,\quad \Gamma :=\overline{\cN}\cap\overline{\cD}.
\end{equation*}
The elliptic operator $L$ and the associated conormal derivative operator $B$ are defined as
\begin{equation}\label{eqn-1201-1421}
\begin{split}
L u &:= D_i(a_{ij}(x)D_j u+b_i(x) u)+\hat{b}_i(x) D_i u+c(x)u,\\
Bu &:= (a_{ij}D_j u+b_i u)n_i,
\end{split}
\end{equation}
where $\vec{n}=(n_i)_{i=1}^d$ is the outer normal direction and $a_{ij}=a_{ji}$. We always assume for some constants $\Lambda\in(0,1]$ and $K>0$,
\begin{equation*}
\Lambda |\xi|^2\le\sum_{i,j=1}^d a_{ij}(x)\xi_j\xi_i\le \Lambda^{-1} |\xi|^2, \quad \forall \xi\in \bR^d\,\,\text{and}\,\, x\in \overline{\Omega},\quad |b_i|+|\hat{b}_i|+|c|\le K.
\end{equation*}

Unlike the purely Dirichlet or conormal boundary value problem, solutions to \eqref{eqn-11291710} can be non-smooth near $\Gamma$ even if the domain, coefficients, and boundary data are all smooth. For example, the function $u={\rm Im}\sqrt{z}$ (with $z=x+iy$) is harmonic in the upper half-plane with zero Dirichlet data on the positive real axis and zero Neumann data on the negative real axis, but u is only in $C^{1/2}$ and $W^{1,4-\varepsilon}$. Our aim is to find minimum smoothness assumptions on $\Omega$ and the separation $\Gamma=\overline{\cN}\cap\overline{\cD}$, such that certain ``optimal regularity'' is achieved.

When $\p\Omega$, $\Gamma$, and all the coefficients are sufficiently smooth, there are quite a few results in the literature concerning the optimal regularity. See, for example, the $W^{1,4-\epsi}$-regularity on half space and the more general result of $W^{s,p}$-regularity on smooth domains in \cite{Sh}. In \cite{Sav}, the optimal $B^{3/2}_{2,\infty}$-regularity on $C^{1,1}$ domains was obtained. For more details and history, see \cite{CDL} and the references therein.

On irregular domains, it turns out that the only requirement for reaching the optimal regularity is certain ``flatness'' of $\p\Omega$ and $\Gamma$. Indeed, for the problem with homogeneous boundary condition, i.e., \eqref{eqn-11291710} with $g_\cN=g_\cD=0$, in \cite{CDL} we proved the $W^{1,4-\epsi}$ regularity for weak solutions, assuming both $\p\Omega$ and $\Gamma$ to be Reifenberg flat (see Assumptions \ref{ass-RF} and \ref{ass-separation}(a)). The boundaries of Reifenberg flat domains are only flat in the sense that they are close to hyperplanes in the Hausdorff distance, not in the sense the regularity. Typically such domains can have fractal structures. Our result in \cite{CDL} is a generalization of the $W^{1,p}$ regularity of purely Dirichlet or conormal problem with ``partially VMO'' coefficients in \cite{DK11} and \cite{DK12}. The first objective of the current paper is to generalize the result in \cite{CDL} by allowing $\Gamma$ to be perturbations of Lipschitz graphs, not just hyperplanes. See Theorem \ref{thm-wellposedness}.

The second objective of the paper is to study the extension problem of the inhomogeneous boundary data $g_\cN$ and $g_\cD$. Stronger than the usual trace spaces, we consider almost everywhere defined boundary data: $g_\cN\in L^q(\cN)$ and $g_\cD\in W^{1,q}(\cD)$. We aim to derive an $L^q$ non-tangential maximal function estimate for the Laplace equation. Such non-tangential maximal function estimate also implies the unique solvability: for any $g_\cN\in L^q(\cN)$ and $g_\cD\in W^{1,q}(\cD)$, we can find a unique harmonic function satisfying the boundary conditions in the sense of ``non-tangential limit''.

In this direction, the study of the above problem is an extension of that of the Neumann problem with $L^q$ data or the Dirichlet problem with $W^{1,q}$ data on Lipschitz domains, for which the results were obtained in \cite{JK81} for $q=2$ and in \cite{DK-87} for $q<2+\epsi$. Note that the range $q<2+\epsi$ is optimal. Due to the failure of the usual $L^2$-method, the corresponding estimate for the mixed problem was raised as an open problem in \cite{Kenig-book}. Initiated in \cite{B}, one approach is to assume an additional geometric condition on the domain such that the optimal regularity of the non-tangential maximal function is above $L^2$, hence the classical $L^2$-method still applies. See also \cite{MR2309180}. Such geometric assumption clearly excludes smooth domains. The other approach started from the $L^1$-solvability with boundary data in Hardy spaces. In this direction, the best result so far is the unique $L^{1+\epsi}$-solvability obtained in \cite{TOB}. For this, they allowed $\Omega$ to be any Lipschitz domain and $\Gamma$ to be very rough: merely Ahlfors regular of Hausdorff dimension close to $d-2$, in addition to the so-called ``corkscrew conditions'' on $\cD$. Following \cite{TOB}, recently an explicit solvability range $q<d/(d-1)$ was obtained in \cite{BC}, assuming $\p\Omega \in C^{1,1}$ and $\Gamma$ to be a Lipschitz graph.

In the current paper, we prove that for $m=0,\ldots,d-2$, the $L^{q}$ non-tangential maximal function estimate holds when $q\in (1,(m+2)/(m+1))$, under the conditions that $\p\Omega$ is Lipschitz and Reifenberg flat, and $\Gamma$ is a perturbation (measured by the Hausdorff distance) of a Lipschitz graph in $m$ variables, $m\in[0,d-2]$. See Theorem \ref{thm-inhomo}. In particular, when $m=d-2$, our result generalized the one in \cite{BC} by allowing rougher $\p\Omega$ and $\Gamma$. In another special case when $m=0$, i.e., $\Gamma$ is Reifenberg flat of co-dimension 2, the optimal range $q<2-\epsi$ is achieved.

In this paper, we only consider the non-tangential maximal function estimate for the Laplace equation. In a subsequent work, we plan to discuss the perturbation theory where elliptic operators with variable coefficients are considered, with the aim to generalize the corresponding result for the conormal problem in \cite{DPR}.

The rest of the paper is organized as follows. In Section \ref{sec-main results}, we will set up the problems and then give our main results: Theorem \ref{thm-wellposedness} for $W^{1,p}$-estimate of weak solutions and Theorem \ref{thm-inhomo} for the $L^q$ non-tangential maximal function estimate. In Sections \ref{sec-halfspace} and \ref{sec-W1p}, we give the proof of Theorem \ref{thm-wellposedness}, where Section \ref{sec-halfspace} is for the estimates on the half space and Section \ref{sec-W1p} is for the perturbation argument. Finally, we prove Theorem \ref{thm-inhomo} in Section \ref{sec-ntm}. In the proof, Corollary \ref{cor-10182245} from the previous homogeneous boundary data part plays a key role in improving the regularity.
\section{Problem set up and main results}\label{sec-main results}
To begin with, let us give all the definitions of domains which we will be discussing throughout this paper.
\begin{assumption}[$M$-Lipschitz]\label{ass-small-Lip}
There exists some constant $R_0>0$, such that, for any $x_0\in \p\Omega$, there exists a Lipschitz function $\psi_0:\bR^{d-1}\rightarrow \bR$ such that in some coordinate system (upto rotation and translation),
\begin{equation*}
\Omega_{R_0}(x_0):= \Omega \cap B_{R_0}(x_0) = \{x \in B_{R_0}(0) : x^1 > \psi_0(x')\}\quad\mbox{and}\quad  |D\psi_0(x')|<M \quad \mbox{a.e. }
\end{equation*}
\end{assumption}
\begin{assumption}[$\gamma$-flat]\label{ass-RF}
There exists a constant $R_1>0$, such that, for any $x_0\in \partial \Omega$ and $R\in (0,R_1]$, in some coordinate system (upto rotation and translation) depending on $x_0$ and $R$, we have
\begin{equation*}	
\{x:x^1>x_0^1+\gamma R\}\cap B_R\subset \Omega_R(x_0)\subset \{x:x^1>x_0^1-\gamma R\}\cap B_R.
\end{equation*}
\end{assumption}
Clearly, $\gamma$-Lipschitz implies $\gamma$-flat. However, the following example shows that Lipschitz and Reifenberg-flat do not imply small Lipschitz.
\begin{example}
We construct a curve by gluing infinitely many congruent copies of a ``S''-shaped curve, while the $k$-th copy is of size $2^{-k}$. The ``S''-shaped curve is designed as follows. We start from a horizontal line segment with length $1/2$. Then on its right end we connect a  line segment with slope $1/m$ and length $1/4$. Repeat this process until we link $m+1$ line segments together, where the $j$th segment has slope $(j-1)/m$ and length $2^{-j}$. The last step is to extend it symmetrically beyond the right end point.

By choosing $m$ large enough, the curve is $\epsi$-flat, while the Lipschitz constant is at least $\tan (\pi/8)$.
\end{example}

As shown in \cite{LMS}, any small Reifenberg flat domain is a $W^{1,p}$-extension domain for every $p\in[1,\infty]$. Hence we have all the Sobolev inequalities up to the first order. In the following, we state a local Poincac\'e inequality for functions vanishing only on part of the boundary, which can be found in \cite[Corollary~3.2~(b)]{CDL}. Throughout the paper, for $\cD\subset \p\Omega$ and $p\in[1,\infty)$, we denote the space $W^{1,p}_\cD$ to be the closure of $C^\infty_c(\overline{\Omega}\setminus\cD)$ under the usual $W^{1,p}$-norm.
\begin{lemma}[Poincar\'e inequality with mixed boundary condition]	\label{lem-poincare}
Let $\gamma\in [0,1/48]$ and $\Omega\subset \bR^d$ be a Reifenberg flat domain satisfying Assumption \ref{ass-RF} $(\gamma)$.
Let $x_0\in \overline{\Omega}$, $R\in (0, R_1/4]$, and $\cD\subset \partial \Omega$ with $\cD\cap B_R(x_0)\neq \emptyset$.
If there exist $z_0\in \cD\cap B_R(x_0)$ and $\alpha\in (0,1)$ such that
\begin{equation*}
B_{\alpha R}(z_0)\subset B_R(x_0), \quad \big(\partial \Omega\cap B_{\alpha R}(z_0)\big) \subset \big(\cD \cap B_R(x_0)\big),
\end{equation*}
then, for any $u\in W^{1,2}_{\cD}(\Omega)$, we have
$$
\|u\|_{L^2(\Omega_R(x_0))}\le CR\|Du\|_{L^2(\Omega_{2R}(x_0))},
$$
where $C=C(d,\alpha)$.
\end{lemma}
The assumptions for the interfacial boundary $\Gamma$ are given as follows.
\begin{assumption}[$(\gamma,m)$-flat separation]\label{ass-separation}
Let $\Omega$ be a domain satisfying either Assumption \ref{ass-small-Lip} or \ref{ass-RF}, with $\p\Omega$ divided into two non-intersecting portions $\cD$ and $\cN$. Let $\Gamma$ be the boundary (relative to $\partial \Omega$) of $\cD$. We call $\Gamma$ a $(\gamma,m)$-flat separation if for any $x_0\in \Gamma$ and $R\in (0, R_1]$, the following holds.
\begin{enumerate}
\item When $m=0$, the coordinate system given in the previous assumption satisfies
$$
\big(\partial \Omega \cap B_R(x_0)\cap \{x: x^2>x_0^2+\gamma R\}\big)\subset \cD,
$$
$$
\big(\partial \Omega \cap B_R(x_0)\cap \{x: x^2<x_0^2-\gamma R\}\big)\subset \cN.
$$
\item When $1\leq m \leq d-2$, there is a Lipschitz function $\phi:\bR^m\rightarrow\bR$ with Lipschitz constant $M$, such that the coordinate system satisfy
$$
\big(\partial \Omega \cap B_R(x_0)\cap \{x: x^2 > \phi(x^3,\ldots,x^{m+2})+\gamma R\}\big)\subset \cD,
$$
$$
\big(\partial \Omega \cap B_R(x_0)\cap \{x: x^2<\phi(x^3,\ldots,x^{m+2})-\gamma R\}\big)\subset \cN.
$$
\end{enumerate}
\end{assumption}

Now we formulate our first main results. Recall the elliptic operator $L$ and its associated conormal derivative operator $B$ in \eqref{eqn-1201-1421}. For any constant $\lambda$, we denote $L_\lambda:=L-\lambda$. The following problem will be discussed:
\begin{equation}		\label{eqn-07261524}
\begin{cases}
L_\lambda u = f+ D_if_i  & \text{in }\, \Omega,\\
Bu = f_i  n_i  & \text{on }\, \cN,\\
u = 0 & \text{on }\, \cD.
\end{cases}
\end{equation}
We consider weak solutions to \eqref{eqn-07261524}.
\begin{definition}[$W^{1,p}_\cD$-weak solution]
                                    \label{def3.16}
For $\Omega, \cN,\cD, \Gamma$ given as before, and $p\in[1,\infty)$, we call $u\in W^{1,p}_\cD$ a weak solution to \eqref{eqn-07261524}, if for any $\zeta\in C^\infty_c(\Omega\cup\cN)$ (or equivalently, $W^{1,p/(p-1)}_{\cD}(\Omega)$),
\begin{equation*}
\int_\Omega (- a_{ij}D_j u - b_i u) D_i \zeta + (\hat{b}_i D_i u + cu)\zeta\,dx=\int_\Omega f  \zeta\,dx-\int_\Omega f_i D_i\zeta\,dx.
\end{equation*}
\end{definition}
Since there is no boundary term in the integral form, the weak solution is still well defined for very rough $\p\Omega$. In particular, we can still discuss weak solution for Reifenberg flat domains where $\vec{n}$ is not point-wise defined and there is no notion of the trace space.

For the leading coefficients, the following small BMO condition will be considered.
\begin{assumption}[$\theta$-BMO]\label{ass-smallBMO}
There exists $R_1\in (0,1]$ such that for any $x\in \overline{\Omega}$ and $r\in (0, R_1]$, we have
 $$
\dashint_{\Omega_r(x)}|a_{ij}(y)-(a_{ij})_{\Omega_r(x)}|\,dy < \theta.
$$
\end{assumption}
In the above assumption and thoughout this paper, we use the following notation for the average:
\begin{equation*}
(f)_{\Omega_r(x)} = \dashint_{\Omega_r(x)} f(y)\,dy := \frac{1}{|{\Omega_r(x)}|} \int_{\Omega_r(x)} f(y)\,dy.
\end{equation*}
In one of our main results, the following sign condition will be assumed: $L_01\leq 0$ is satisfied in the weak sense, if
\begin{equation*}
\int_\Omega (-b_iD_i\zeta + c\zeta)\leq 0,\quad \forall \zeta \in C_c^\infty(\Omega),\,\,\zeta\geq 0.
\end{equation*}
We also denote the following index $p_{*}$:
\begin{equation*}p_{*}:=
\begin{cases}
pd/(p+d)\,\,&\text{when}\,\, p>d/(d-1),\\
1+\epsi \,\,&\text{when}\,\, p\leq d/(d-1),
\end{cases}
\end{equation*}
where $\epsi$ can be any positive constant. In this paper, for simplicity, when $\lambda>0$ we use the following notation
\begin{equation*}
U:=|Du| + \sqrt{\lambda}|u|,\quad F:= \sum_i|f_i| + \frac{1}{\sqrt{\lambda}}|f|.
\end{equation*}
Now we state our first main result regarding equations with homogeneous boundary conditions.
\begin{theorem}\label{thm-wellposedness}
For any integer $m\in[0,d-2]$ and constant $p\in (2(m+2)/(m+3),2(m+2)/(m+1))$, we can find positive constants $(\gamma_0,\theta_0)=(\gamma_0,\theta_0)(d,p,M,\Lambda)$, such that if Assumptions \ref{ass-RF} $(\gamma_0)$, \ref{ass-separation} $(\gamma_0,m)$, and \ref{ass-smallBMO} $(\theta_0)$ are satisfied, the following hold.\\
(i) There exists $\lambda_0=\lambda_0(d,p, M, \Lambda, R_1, K)$ such that for any $(f_i)_{i=1}^d\in (L^p(\Omega))^d$ and $f\in L^p(\Omega)$, the problem \eqref{eqn-07261524} with $\lambda>\lambda_0$ has a unique solution $u\in W^{1,p}_\cD(\Omega)$. The solution satisfies the estimate
\begin{equation*}
\norm{U}_{L^p(\Omega)} \leq C \norm{F}_{L^p(\Omega)},
\end{equation*}
where $C=C(d,p,M,\Lambda)$ is a constant.\\
(ii) On a bounded domain $\Omega$, if $L_01\leq 0$ in the weak sense, then for any $(f_i)_{i=1}^d\in (L^p(\Omega))^d$ and $f\in L^{p_{*}}(\Omega)$, \eqref{eqn-07261524} with $\lambda=0$ has a unique solution $u\in W^{1,p}_\cD(\Omega)$ satisfying
\begin{equation*}
\norm{u}_{W^{1,p}(\Omega)} \leq C\Bigg(\sum_{i=1}^d\norm{f_i}_{L^p(\Omega)} + \norm{f}_{L^{p_{*}}(\Omega)}\Bigg),
\end{equation*}
where $C$ is a constant independent of $u$, $f$, and $f_i$.
\end{theorem}
The following example in \cite{CDL} shows that the symmetry of $\vec A$ is required for the $W^{1,p}$-regularity when $p$ is away from $2$.
\begin{example}
In $\bR^2_+:=\{(x,y):y>0\}$, let
$u(x,y)={\rm Im}(x+iy)^s$ with $s \in (0,1/2)$. We have
$$D_i(a_{ij}D_j u)=0 \text{ on } \bR^2_{+},\quad u=0 \text{ on } \p\bR^2_{+}\cap\set{x>0},\quad a_{ij}D_j u n_i=0 \text{ on }\p\bR^2_{+}\cap\set{x<0},
$$
where
$$(a_{ij})_{i,j=1}^{2} = \begin{bmatrix}
1 & \cot(\pi s)\\-\cot(\pi s) & 1
\end{bmatrix}.
$$
Since $Du$ is of order $r^{s-1}$, near the origin we have $Du\in L^p$ only for $p<\frac{2}{1-s}$. Note that $\frac{2}{1-s}<4$ and $\frac{2}{1-s}\searrow 2$ as $s\searrow 0$.
\end{example}

Next we formulate the problem with inhomogeneous boundary data. To make sense of the boundary condition, we introduce the following concepts.
\begin{definition}
Let $\beta>0$ be a constant. For a domain $\Omega$ and a point $x\in \p\Omega$, we define the non-tangential approach region at $x$ with opening $\beta$ as
\begin{equation*}
\Gamma_\beta(x):=\{y\in\Omega:|x-y|\leq (1+\beta) \dist(y,\p\Omega)\}.
\end{equation*}
The corresponding non-tangential maximal function is defined as
\begin{equation*}
\vec{N}(u)(x):=\sup_{Y\in\Gamma_\beta(x)} |u(y)|.
\end{equation*}
In this paper, we will omit the dependence of the opening $\beta$.
\end{definition}
For some $q\in[1,\infty)$, we consider the following boundary value problem
\begin{equation}		\label{190527@eq2}
\begin{cases}
\Delta u = 0  & \text{in }\, \Omega,\\
\frac{\p u}{\p \vec{n}} = g_\cN  & \text{on }\, \cN,\\
u = g_\cD & \text{on }\, \cD,\\
\vec{N}(Du) \in L^q(\p\Omega).
\end{cases}
\end{equation}
For $g_\cN\in L^q(\cN)$ and $g_\cD\in W^{1,q}(\cD)$, we call $u$ a solution to \eqref{190527@eq2}, if $u\in W^{1,2}(\p\Omega)$ is a weak solution in the usual sense, and $\vec{N}(Du)$ is controlled. For such solution, the Dirichlet boundary condition is satisfied in the sense of non-tangential limit and the conormal boundary condition is satisfied in the sense of ``$L^q$-weak non-tangential limit''. See, for instance, \cite[Theorem~1.8.1]{Kenig-book}.

The second main result of the paper is the following theorem.
\begin{theorem}\label{thm-inhomo}
Let $\Omega\subset \bR^d$ be a bounded domain satisfying Assumption \ref{ass-small-Lip}$(M)$. For any $q\in(1,\frac{m+2}{m+1})$, we can find sufficiently small $\gamma_1=\gamma_{1}(d,q,M)>0$, such that if $\Omega$ and $\Gamma$ satisfy the Assumptions \ref{ass-RF} $(\gamma_1)$ and \ref{ass-separation} $(\gamma_1,m)$, then for any $g_\cN\in L^q(\cN), g_\cD\in W^{1,q}(\cD)$, there exists a unique solution $u$ to the problem \eqref{190527@eq2} satisfying
\begin{equation*}
\norm{\vec{N}(Du)}_{L^q(\p\Omega)}\le  C\norm{g_\cN}_{L^q(\cN)} + C\norm{g_\cD}_{W^{1,q}(\cD)},
\end{equation*}
where $C=C(d,M,R_0,R_1,\text{diam}(\Omega),q)$ is a constant.
\end{theorem}

\section{Harmonic function on half space}\label{sec-halfspace}
Let $\phi$ be a Lipschitz function $\bR^m\rightarrow \bR$ with $\phi(0)=0$. Let
$$
\bR^{d}_{+}:=\{x^1>0\},\quad B^{+}_R:=B_R(0)\cap \{x^1>0\}.
$$
Throughout this section, for $m=0,\ldots,d-2$, we use the notation
\begin{equation*}
\Gamma:= \{x^1=0, x^2=\phi(x^3,\ldots,x^{m+2})\},\quad \cD:=\{x^1=0,x^2>\phi\},\quad \cN:=\{x^1=0,x^2<\phi\}.
\end{equation*}
When $m=0$, we just set $\phi=0$.
The main result in this section is the following reverse H\"older inequality.
\begin{theorem}\label{thm-halfspace}
Let $u\in W^{1,2}(B^{+}_R)$ be a weak solution to the following problem with $\lambda>0$:
\begin{equation*}
\begin{cases}
-\Delta u + \lambda u= 0 & \text{in }B^{+}_R,\\
\frac{\p u}{\p x^1}=0 & \text{on } \cN\cap B_R,\\
u=0 &\text{on } \cD\cap B_R,
\end{cases}
\end{equation*}
Then for any $p<\frac{2(m+2)}{m+1}$, we have $u\in W^{1,p}(B_{R/2}^{+})$ satisfying
\begin{equation*}
(U^p)^{1/p}_{B_{R/2}^{+}} \leq C(U^2)^{1/2}_{B_{R}^{+}},
\end{equation*}
where $C=C(d,p,M)$.
\end{theorem}
The case $m=0$ was proved in \cite[Theorem~4.1]{CDL}. In the following, we first prove a lemma which contains the result when $m=d-2$. From this lemma, we can derive the case when $m$ is in between. In this section, we do not distinguish the geometric objects on $\bR^{m+2}$ or $\bR^d$.
\begin{lemma}\label{lem-halfspace-(m+2)d}
On $\bR^{m+2}$, let $u\in  W^{1,2}(B^{+}_R)$ be a weak solution to
\begin{equation}\label{eqn-07251828}
\begin{cases}
-\Delta u = f & \text{in }B^{+}_1,\\
\frac{\p u}{\p x^1}=0 & \text{on } \cN\cap B_1,\\
u=0 &\text{on } \cD\cap B_1,
\end{cases}
\end{equation}
where $f\in L^2(B^{+}_1)$. Then for any $p\in \big[2,\frac{2(m+2)}{m+1}\big)$, we have $u\in W^{1,p}(B^{+}_{1/2})$ and
\begin{equation*}
\norm{\nabla u}_{L^p(B^{+}_{1/2})} \leq C(m, p)( \norm{\nabla u}_{L^2(B^{+}_1)} + \norm{f}_{L^2(B^{+}_1)}).
\end{equation*}
\end{lemma}
\begin{proof}
The case when $m=0$ is proved in \cite[Lemma~4.3]{CDL}. Here we only consider the case when $m\geq 1$. In the following, we construct a problem on the half space $\bR^{m+2}_{+}$, which coincides with \eqref{eqn-07251828} in $B_{1/2}^{+}$. Hence the Besov space regularity in \cite[Theorem~3.3]{BC} applies. For this, first we take zero extension of $u$ on $\bR^{m+2}_{+}\setminus B_1^{+}$. Then let
$$\eta \in C^\infty_c(B_1),\quad \eta =1\,\,\text{in }B_{1/2},\quad \text{even in}\,\, x^1.$$
Now $u\eta\in W^{1,2}(\bR^{m+2}_{+})$, and solves
\begin{equation*}
\begin{cases}
-\Delta (u\eta) = f\eta - 2\nabla u\cdot\nabla\eta -u\Delta\eta & \text{in}\,\,\bR_+^{m+2},\\
\frac{\p (u\eta)}{\p x^1}=0 & \text{on}\,\, \cN,\\
u\eta=0 &\text{on}\,\, \cD.
\end{cases}
\end{equation*}
We have
\begin{align}
&\norm{\nabla u}_{L^p(B^{+}_{1/2})} = \norm{\nabla (u\eta)}_{L^p(B^{+}_{1/2})}\nonumber \\ &\lesssim \norm{\nabla (u\eta)}_{B^{1/2}_{2,\infty}(\bR^{m+2}_{+})}\label{est-07261420-1}\\
&\lesssim \norm{f\eta -2 \nabla u\cdot \nabla \eta - u\Delta\eta}_{L^2(\bR^{m+2}_{+})}\label{est-07261420-2}\\
&\lesssim \norm{f}_{L^2(B_1^{+})} + \norm{u}_{W^{1,2}(B_1^{+})}\nonumber\\
&\lesssim \norm{f}_{L^2(B_1^{+})} + \norm{\nabla u}_{L^2(B_1^{+})}\label{est-07261420-3}.
\end{align}
Here, \eqref{est-07261420-1} is due to the Besov-Sobolev embedding, where $p<\frac{2(m+2)}{m+1}$ is required, \cite[Theorem~3.3]{BC} is applied to obtain \eqref{est-07261420-2}, and the estimate \eqref{est-07261420-3} is obtained by applying the usual Poincar\'e inequality on Lipschitz domains with the homogeneous Dirichlet condition imposed on part of the boundary (see Lemma \ref{lem-poincare}).
\end{proof}
Now we turn to the proof of the main result of this section.
\begin{proof}[Proof of Theorem \ref{thm-halfspace}]
As mentioned before, we only need to prove when $0<m<d-2$. For simplicity, here we introduce the notation $x=(\bar{x},\hat{x})$, where $\bar{x}:=(x^1,\ldots,x^{m+2})$ and $\hat{x}:=(x^{m+3},\ldots,x^d)$, and for $p,q\in [1,\infty]$ we denote
$$
\|u\|_{L^p_{\bar{x}}L^q_{\hat{x}}}=\big\|\|u\|_{L^q_{\hat{x}}}\big\|_{L^p_{\bar{x}}}.
$$
The proof is similar to that of \cite[Theorem~4.1]{CDL}. By applying Agmon's idea as in the proof of \cite[Theorem~4.1]{CDL} and scaling, it suffices to prove when $\lambda=0, R=1$. Also, since the case when $p\leq 2$ is trivial, in the following we assume $2 < p < \frac{2(m+2)}{m+1}$.

We first rewrite \eqref{eqn-07251828} as an equation in the $\bar{x}$-variables: for every $|\hat{x}|<1/2$,
$$\begin{cases}
\Delta_{\bar{x}}u(\cdot,\hat{x})= -\Delta_{\hat{x}}u(\cdot,\hat{x}) & \text{in }B_{2/3}^{+},\\
\frac{\p u}{\p x^1}=0 & \text{on } \cN\cap B_{2/3},\\
u=0 &\text{on } \cD\cap B_{2/3}.
\end{cases}$$
Now, a properly rescaled version of Lemma \ref{lem-halfspace-(m+2)d} leads to
\begin{equation*}
\|D_{\bar{x}}u(\cdot,\hat{x})\|_{L^p(B^{+}_{1/2})} \lesssim \|D_{\bar{x}}u(\cdot,\hat{x})\|_{L^2(B_{2/3}^+)}
+\|\Delta_{\hat{x}}u(\cdot,\hat{x})\|_{L^2(B_{2/3}^+)},\quad\forall\  |\hat{x}|<1/2.
\end{equation*}
Taking $L^p$ norm for $\hat{x}\in B_{1/2}$, we obtain
\begin{equation}\label{est-07261502}
\|D_{\bar{x}}u\|_{L^p(B_{1/2}^{+})} \lesssim \|D_{\bar{x}}u\|_{L^p_{\hat{x}}L^2_{\bar{x}}(B_{3/4}^+)}
+\|\Delta_{\hat{x}}u\|_{L^p_{\hat{x}}L^2_{\bar{x}}(B_{3/4}^+)}.
\end{equation}
Due to the translation invariance in $\hat{x}$, we can differentiate both the equation and boundary conditions in the $\hat{x}$-direction. Hence we obtain the following estimate by applying the Caccioppoli inequality $k$ times:
\begin{equation*}
\norm{DD^k_{\hat{x}}u}_{L^2(B^{+}_s)} \leq \frac{C(d,k)}{|t-s|^k}\norm{Du}_{L^2(B_t^{+})}
\end{equation*}
for $0<s<t\leq 1$ and $k\in\{0,1,2,\ldots\}$. From this and the anisotropic Sobolev embedding (see, for instance, \cite[Sec. 18.12]{MR521808}),
\begin{equation*}
\|D_{\bar{x}}u\|_{L^2_{\bar{x}}L^p_{\hat{x}}(B^+_{r})}
+\|D_{\hat{x}}u\|_{L^{p}(B^+_{r})}
+\|D_{\hat{x}}^2u\|_{L^{p}(B^+_{r})}
\le C(d,p,r)\|Du\|_{L^2(B^+_1)},\quad \forall r\in(0,1),
\end{equation*}
where $p\in (2,2(m+2)/(m+1))$. Combining this, \eqref{est-07261502}, and the Minkowski inequality, we reach the desired estimate.
\end{proof}

\section{Proof of Theorem \ref{thm-wellposedness}: regularity of $W^{1,2}$-weak solutions}\label{sec-W1p}
The key step in proving Theorem \ref{thm-wellposedness} is the regularity result Proposition \ref{prop-regularity} stated below. The following equation with $b_i=\hat{b}_i=c=0$ will be discussed:
\begin{equation}\label{eqn-07282306}
\begin{cases}
L_\lambda^{(0)}u := D_i(a_{ij}D_j u)-\lambda u= f+ D_i f_i  & \text{in }\, \Omega,\\
a_{ij}D_j u n_i = f_i  n_i & \text{on }\, \cN,\\
u=0 & \text{on }\, \cD.
\end{cases}
\end{equation}
\begin{proposition}\label{prop-regularity}
For any $p\in (2,\frac{2(m+2)}{m+1})$, we can find positive constants $\gamma_0$ and $\theta_0$ depending on $(d,p,M,\Lambda)$, such that if Assumptions \ref{ass-RF} $(\gamma_0)$, \ref{ass-separation} $(\gamma_0,m)$, and \ref{ass-smallBMO} $(\theta_0)$ are satisfied, the following holds.
For any $W^{1,2}_{\cD}(\Omega)$ weak solution $u$ to \eqref{eqn-07282306} with $\lambda>0$ and $f_i, f\in L^p(\Omega)\cap L^2(\Omega)$,
 we have $u \in W^{1,p}_{\cD}(\Omega)$ and
\begin{equation}\label{est-no-lower-order}
\norm{U}_{L^p(\Omega)} \leq C \big(R_1^{d(1/p-1/2)}\norm{U}_{L^2(\Omega)} + \norm{F}_{L^p(\Omega)}\big).
\end{equation}
Furthermore, if we also assume $f\equiv 0$, then we can take $\lambda=0$. In this case, the following estimate holds:
\begin{equation*}
\norm{Du}_{L^p(\Omega)} \leq C\big(R_1^{d(1/p-1/2)}\norm{Du}_{L^2(\Omega)} + \norm{f_i}_{L^p(\Omega)}\big).
\end{equation*}
In the above, the constant $C$  depends on $d$, $p$, $M$, and $\Lambda$.
\end{proposition}

We give the following corollary which will be useful for the inhomogeneous boundary condition case later.
\begin{corollary}\label{cor-10182245}
Let $0<r<R\le \text{diam}(\Omega)$.
Consider a harmonic function in $\Omega_R$ with the corresponding mixed boundary conditions on $\p\Omega\cap B_R$. Under the assumptions for $p,\Omega$, and $\Gamma$ as in Proposition \ref{prop-regularity}, we have for any $r<R$,
\begin{equation*}
\norm{Du}_{L^p(\Omega_r)}\le  C(R-r)^{-d/p'}\norm{Du}_{L^1(\Omega_R)},
\end{equation*}
where $p'$ satisfies $1/p+1/p'=1$ and $C=C(d,p,M,\Lambda,R_1)$.
\end{corollary}
The proof is standard, which we will only sketch here. We may certainly assume that $r\ge R/2$. By using a covering argument (with balls of radius $(R-r)/2$ centered at, say $x_1,\ldots,x_k\in \Omega_r$), we first localize the result in Proposition \ref{prop-regularity} to obtain an $L^2-L^p$ estimate: for $j=1,\ldots,k$,
\begin{equation*}
 (|Du|^p)^{1/p}_{\Omega_{(R-r)/4}(x_j)} \leq C\big( (|Du|^2)^{1/2}_{\Omega_{(R-r)/2}(x_j)} + (|u|^2)^{1/2}_{\Omega_{(R-r)/2}(x_j)}\big).
\end{equation*}
To remove the $L^2$-norm of $u$ on the right-hand side, we apply the Poincar\'e inequality. When $\Omega_{(R-r)/2}(x_j)\cap \cD=\emptyset$, we replace $u$ with $u-(u)_{\Omega_{(R-r)/2}(x_j)}$, which locally satisfies the same equation with the conormal boundary condition, and then apply the usual Poincar\'e inequality. Otherwise, we need to apply the Poincar\'e inequality in Lemma \ref{lem-poincare}. Finally, to replace $\norm{Du}_{L^2}$ with $\norm{Du}_{L^1}$, we use H\"older's inequality and an iteration argument. Summing in $j$, we get the desired estimate.

The remainder of this section will be mostly devoted to the proof of Proposition \ref{prop-regularity}. We first sketch the idea. By constructing a cut-off function and applying a reflection technique, at all small scales we decompose the solution into two parts (Lemma \ref{lem-decomposition}). One part (up to rotation) solves the problem in Theorem \ref{thm-halfspace}, hence can reach the optimal $W^{1,\frac{2(m+2)}{m+1}-\epsi}$-regularity. The other part deals with all the perturbation terms measured by a $W^{1,2}$-estimate.

Next we ``interpolate'' these two part to reach the regularity of $u$ in between, by applying the level set argument introduced in \cite{CP}. The key idea is to use a measure theoretical lemma called ``crawling of the ink spots'' in \cite{KS}: from the decomposition and a Chebyshev-type inequality, we first deduce a local property at certain small scales $R$, regarding the level set of the maximal function (Lemma \ref{lem-levelset-local}). Then the ``crawling of the ink spots'' lemma leads to a global decay estimate of the measure of the level sets (Corollary \ref{cor-levelset-global}). From this decay estimate, the desired $L^p$-estimate can be obtained by an integral representation of the $L^p$ norm and the Hardy-Littlewood maximal function theorem.

The proof follows similar steps as in \cite[Section~5]{CDL}, where detailed computation can be found. Essentially the new ingredient here is the construction of a cut-off function when proving Lemma \ref{lem-decomposition}.
\subsection{A reverse H\"older inequalities}
Using the local Sobolev-Poincar\'e inequality in Lemma \ref{lem-poincare} and Gehring's Lemma, we first improve the regularity of a weak solution from $W^{1,2}$ to $W^{1,2+\epsi}$.
\begin{proposition}[Reverse H\"older's inequality]		\label{prop-reverse-holder}
Let $\gamma\in (0, 1/48]$, $p>2$, and the integer $m\in[0,d-2]$. Assume on $\bR^d$, $\Omega$, and $\Gamma$ satisfy Assumptions \ref{ass-RF} $(\gamma)$ and \ref{ass-separation} $(\gamma,m)$, and the function $u\in W^{1,2}_\cD(\Omega)$ satisfies \eqref{eqn-07282306} with $f_i, f\in L^p(\Omega)\cap L^2(\Omega)$.
Then there exist constants $p_0\in (2,p)$ and $C>0$, depending only on $d$, $p$, M, and $\Lambda$, such that for any $x_0\in \bR^d$ and $R\in (0, R_1]$, the following hold. When $\lambda>0$, we have
$$
\big(\overline{U}^{p_0}\big)^{1/p_0}_{B_{R/2}(x_0)}\le C \big(\overline{U}^{2}\big)^{1/2}_{B_{R}(x_0)}+C\big(\overline{F}^{p_0}\big)^{1/p_0}_{B_{R}(x_0)}.
$$
When $\lambda=0$ and $f\equiv 0$, we have
$$
\big(|D\overline{u}|^{p_0}\big)^{1/p_0}_{B_{R/2}(x_0)}\le C \big(|D\overline{u}|^{2}\big)^{1/2}_{B_{R}(x_0)}+C\big(|\overline{f_i}|^{p_0}\big)^{1/p_0}_{B_{R}(x_0)},
$$
where $\overline{U}$, $\overline{F}$, $D\overline{u}$, and $\overline{f_i}$ are the zero extensions of $U$, $F$, $Du$, and $f_i$ to $\bR^d$.
\end{proposition}
The proof can be simply adapted from the one of \cite[Lemma~3.4]{CDL}, which we omit here.
\subsection{Decomposition of solutions}
In this subsection, we prove the following lemma.
\begin{lemma}\label{lem-decomposition}
Suppose that $u\in W^{1,2}_{\cD}(\Omega)$ satisfies \eqref{eqn-07282306}
with $\lambda>0$ and $f_i,f\in L^p(\Omega)\cap L^2(\Omega)$, where $p>2$.
Then under Assumptions \ref{ass-RF} $(\gamma)$, \ref{ass-separation} $(\gamma,m)$, and \ref{ass-smallBMO} $(\theta)$ with $\gamma<1/(32\sqrt{d+3})$ and $\theta\in (0,1)$, for any $x_0 \in \overline{\Omega}$ and $R<R_1$, there exist nonnegative functions $W, V \in L^2(\Omega_{R/32}(x_0))$ such that
$$U \leq W+V \quad \text{in } \Omega_{R/32}(x_0).$$
Moreover, we have for any $q < \frac{2(m+2)}{m+1}$,
\begin{align}
(W^2)^{1/2}_{\Omega_{R/32}(x_0)} &\leq C\big( (\theta^{\frac{1}{2\mu'}}+ \gamma^{\frac{1}{2\mu'}})(U^2)^{1/2}_{\Omega_R(x_0)} + (F^{2\mu})^{\frac{1}{2\mu}}_{\Omega_R(x_0)}\big),\label{eqn-est-W}\\
(V^q)^{1/q}_{\Omega_{R/32}(x_0)} &\leq C\big( (U^2)^{1/2}_{\Omega_R(x_0)} + (F^{2\mu})^{\frac{1}{2\mu}}_{\Omega_R(x_0)}\big).\label{eqn-est-V}
\end{align}
Here $\mu$ is a constant satisfying $2\mu=p_0$, where $p_0=p_0(d,p,M,\Lambda)>2$ comes from Proposition \ref{prop-reverse-holder}, and $\mu'$ satisfies $1/\mu+1/\mu'=1$.
The constant $C$ only depends on $d$, $p$, $q$, M, and $\Lambda$.
\end{lemma}
\begin{proof}[Proof of Lemma \ref{lem-decomposition}]
When $\dist(x_0,\Gamma)\geq R/16$, we either only deal with the interior case or the case with purely Dirichlet/conormal boundary condition (depending on $B_{R/16}(x_0)\cap\cN =\emptyset$ or $B_{R/16}(x_0)\cap\cD =\emptyset$). The interior $W^{1,p}$-estimates for equations with VMO coefficients are by now standard, while the corresponding estimates for purely Dirichlet/conormal problems on Reifenberg flat domains can be found in \cite{DK11,DK12}. Also, one may refer to \cite[page~22]{CDL}. In the following, we focus on the case when $\dist(x_0,\Gamma)< R/16$, where the boundary condition is ``mixed''.

Pick a point $y_0\in\Gamma$ with $\dist(y_0,x_0)<R/16$. Now we take the coordinate system associated with $y_0$ and $R/4$ as in Assumptions \ref{ass-RF} and \ref{ass-separation}, so that $y_0=0$. Denote
\begin{equation*}
\begin{split}
&\Omega_{R/4}:=\Omega\cap B_{R/4},\\
&\Omega^{+}_{R/4}:=\Omega_{R/4}\cap \{x^1>\gamma R/4\},\quad \Omega^{-}_{R/4}:=\Omega_{R/4}\cap \{x^1<\gamma R/4\},\\
&\Gamma^{+}:=\{x^1=\gamma R/4,x^2>\phi-\gamma R/4\}\cap B_{R/4},\\ &\Gamma^{-}:=\{x^1=\gamma R/4,x^2<\phi -\gamma R/4\}\cap B_{R/4}.
\end{split}
\end{equation*}
We take
\begin{equation}\label{eqn-08041607}
W := |Du|+\sqrt{\lambda}|u|\quad \text{and} \quad V:= 0 \quad \text{on}\,\,\Omega_{R/4}^{-}.
\end{equation}
Due to H\"older's inequality, Assumption \ref{ass-RF}, and Proposition \ref{prop-reverse-holder}, we obtain
\begin{align}
(W^2\bI_{\Omega_{R/4}^{-}})^{1/2}_{\Omega_{R/32}(x_0)} &\leq \frac{|\Omega_{R/4}^{-}\cap\Omega_{R/32}(x_0)|^{1/(2\mu')}}
{|\Omega_{R/32}(x_0)|^{1/(2\mu')}}|(W^{2\mu}\bI_{\Omega_{R/4}^{-}})^{1/(2\mu)}_{\Omega_{R/32}(x_0)}\nonumber\\
&\leq C(d,M)\gamma^{1/(2\mu')} (U^{2\mu})^{1/(2\mu)}_{\Omega_{R/32}(x_0)}\nonumber\\
&\leq C(d,p,M,\Lambda)\gamma^{1/(2\mu')} \big((U^2)^{1/2}_{\Omega_R(x_0)} + (F^{2\mu})^{1/(2\mu)}_{\Omega_R(x_0)}\big).\label{est-08051625-1}
\end{align}
In the following, we mainly focus on constructing $W$ and $V$ on $\Omega_{R/4}^{+}$. For this, we introduce a cut-off function $\chi\in C^\infty(\bR^d)$:
\begin{equation*}
\begin{cases}
\chi=0\quad \text{on}\,\,\{x^2>\phi-\gamma R/2\}\cap\{x^1<\gamma R/2\},\\
\chi=1\quad \text{on}\,\,\{x^2<\phi-\gamma R\}\cup\{x^1>2\gamma R\},\\
0\leq\chi\leq 1,\quad |D\chi|\leq \frac{4\sqrt{1+M^2}}{\gamma R}.
\end{cases}
\end{equation*}
We have the following two estimates:
\begin{align}
(|(1-\chi) U|^2)^{1/2}_{\Omega_{R/4}} &\leq C_1 \gamma^{1/(2\mu')}\big((U^2)^{1/2}_{\Omega_R(x_0)} + (F^{2\mu})^{1/(2\mu)}_{\Omega_R(x_0)}\big),\label{est-08041538-1}\\
(|uD\chi|^2)^{1/2}_{\Omega_{R/4}} &\leq C_2 \gamma^{1/(2\mu')}(U^2)^{1/2}_{\Omega_R(x_0)},\label{est-08041538-2}
\end{align}
where $C_1=C_1(d,p,\Lambda)$ and $C_2=C_2(d,p,M,\Lambda)$. The estimate \eqref{est-08041538-1} is a direct consequence of H\"older's inequality and Proposition \ref{prop-reverse-holder}, by noting that
\begin{equation*}
|\Omega_{R/4}\cap\supp(1-\chi)| \leq C(d)\gamma R^d.
\end{equation*}
The estimate \eqref{est-08041538-2} can be obtained as follows: we first decompose
$$
\supp\{D\chi\}\cap\Omega_{R/4}\subset\bigcup_{z\in \cD_{grid}}\Omega_{2\sqrt{d+3}\gamma R}(z),
$$
where
\begin{align*}
&\cD_{grid} := \Big\{z\in \bR^d : z=(\gamma R/4,k\gamma R)\ \text{ for }\ k=(k_2,\ldots,k_d) \in \mathbb{Z}^{d-1},\\
&\qquad \text{where}\ \Omega_{\sqrt{d+3}\gamma R}(z) \cap \{x^2>\phi-3\gamma R/2\} \cap \Omega_{R/4}\neq \emptyset\Big\}.
\end{align*}
Note that on each $\Omega_{2\sqrt{d+3}\gamma R}(z)$, the conditions of Lemma \ref{lem-poincare} are satisfied with a uniform $\alpha\geq\alpha(d,M)>0$. Hence we can apply Lemma \ref{lem-poincare} to obtain
\begin{equation*}
\norm{uD\chi}_{L^2(\Omega_{2\sqrt{d+3}\gamma R}(z))} \lesssim \frac{1}{\gamma R}\norm{u}_{L^2(\Omega_{2\sqrt{d+3}\gamma R}(z))} \lesssim \norm{Du}_{L^2(\Omega_{4\sqrt{d+3}\gamma R}(z))}.
\end{equation*}
Using $\gamma<1/(32\sqrt{d+3})$ and the definition of $\cD_{grid}$, we have $\Omega_{4\sqrt{d+3}\gamma R}(z)\subset \Omega_R(x_0)$. Since each point is covered by at most $N(d)$ such balls, \eqref{est-08041538-2} is proved.

A straightforward but tedious calculation gives the following equation for $\chi u$ on $\Omega_{R/4}^{+}$.
\begin{equation}\label{eqn-08041616}
\begin{cases}
D_i(a_{ij}D_j(\chi u)) - \lambda \chi u= D_ig_i^{(1)} + D_ig_i^{(2)} + g_i^{(3)}D_i\chi + g_i^{(4)}D_i\widetilde{\chi} + g^{(5)} & \text{in } \Omega_{R/4}^{+},\\
a_{ij}D_j(\chi u)n_i = g_i^{(1)}n_i +g_i^{(2)}n_i & \text{on } \Gamma^{-},\\
\chi u=0 &\text{on } \Gamma^{+},
\end{cases}
\end{equation}
where
\begin{align*}
g_i^{(1)}&=a_{ij}uD_j\chi + f_i\chi ,\quad g_i^{(2)}=(-\epsi_{i}\epsi_{j}\widetilde{a_{ij}}\widetilde{\chi}D_j\widetilde{u} + \epsi_{i}\widetilde{\chi}\widetilde{f_i})\bI_{Ex\in\Omega_{R/4}^{-}},\\
g_i^{(3)}&=a_{ij}D_ju-f_i,\quad g_i^{(4)}=(\epsi_i\epsi_j  \widetilde{a_{ij}}D_j\widetilde{u}-\epsi_i\widetilde{f_i})\bI_{Ex\in\Omega_{R/4}^{-}},\\
g^{(5)}&= \chi f  + \widetilde{\chi}\widetilde{f}\bI_{Ex\in\Omega_{R/4}^{-}}+ \lambda\widetilde{\chi}\widetilde{u}\bI_{Ex\in\Omega_{R/4}^{-}}.
\end{align*}
Here we denote the reflection operator and reflected function as
$$E(x^1,x^2,\ldots,x^d):=(\gamma R/2-x^1,x^2,\ldots,x^d),\quad \widetilde{f}:=f\circ E,
$$
and similarly for $\widetilde{a_{ij}}, \widetilde{\chi}$, and $\widetilde{u}$. We also use the following notation for sign functions
$$\epsi_{i}:=\begin{cases}
-1& \text{if }\,\, i=1,\\
1&\text{if }\,\, i\neq 1.
\end{cases}$$
The function $\chi u$ satisfies \eqref{eqn-08041616} in the weak sense: the usual integral identity is satisfied if we take any test function $\psi\in W^{1,2}_{\p\Omega_{R/4}^{+}\setminus\Gamma^{-}}(\Omega_{R/4}^{+})$.

Now we are in position of constructing the decomposition. By the Lax-Milgram lemma, the following equation has a unique solution $w\in W^{1,2}_{\p\Omega_{R/4}^{+}\setminus\Gamma^{-}}(\Omega^{+}_{R/4})$:
\begin{equation}\label{eqn-08041706}
\begin{cases}
\begin{aligned}
D_{i}(\overline{a_{ij}}D_jw) - \lambda w&=
D_i((\overline{a_{ij}}-a_{ij})D_j(\chi u)) + D_ig_i^{(1)} + D_i g_i^{(2)}\\
&\quad + g_i^{(3)}D_i\chi + g_i^{(4)}D_i\widetilde{\chi} +g^{(5)}
\end{aligned}
&\text{in }\Omega_{R/4}^{+},\\
\overline{a_{ij}}D_jw\cdot n_i = (\overline{a_{ij}}-a_{ij})D_j (\chi u) n_i + g_i^{(1)}n_i +g_i^{(2)}n_i &\text{on }\Gamma^{-},\\
w=0 & \text{on }\p\Omega_{R/4}^{+}\setminus\Gamma^{-},
\end{cases}
\end{equation}
where we take $\overline{a_{ij}}:=(a_{ij})_{\Omega_{R/4}}$. Taking the difference between \eqref{eqn-08041616} and \eqref{eqn-08041706}, we obtain that $v:=\chi u-w$ satisfies
\begin{equation*}
\begin{cases}
D_{i}(\overline{a_{ij}}D_jv) - \lambda v= 0 &\text{in }\Omega_{R/4}^{+},\\
\overline{a_{ij}}D_jv \cdot n_i = 0 &\text{on }\Gamma^{-},\\
v=0 &\text{on }\Gamma^{+}.
\end{cases}
\end{equation*}
We define
\begin{equation}\label{eqn-08041725}
W:=|Dw| + \sqrt{\lambda}|w| + |D((1-\chi)u)| + \sqrt{\lambda}|(1-\chi)u|,\quad V:= |Dv| + \sqrt{\lambda}|v|\quad \text{on}\,\,\Omega_{R/4}^{+}.
\end{equation}
By our construction \eqref{eqn-08041607} and \eqref{eqn-08041725}, clearly we have $U\leq W+V$ on $\Omega_{R/32}(x_0)\subset\Omega_{R/4}$. Now we estimate $W$. By \eqref{est-08041538-1} and \eqref{est-08041538-2}, we have
\begin{align}
&(|D((1-\chi)u)|^2)^{1/2}_{\Omega_{R/32}(x_0)} + \sqrt{\lambda}(|(1-\chi)u|^2)^{1/2}_{\Omega_{R/32}(x_0)} \nonumber\\
&\lesssim (|D((1-\chi)u)|^2)^{1/2}_{\Omega_{R/4}} + \sqrt{\lambda}(|(1-\chi)u|^2)^{1/2}_{\Omega_{R/4}}\nonumber\\ &\lesssim \gamma^{1/(2\mu')}\big((U^2)^{1/2}_{\Omega_R(x_0)} + (F^{2\mu})^{1/(2\mu)}_{\Omega_R(x_0)}\big).\label{est-08051625-2}
\end{align}
The estimate for $|Dw|+\sqrt{\lambda}|w|$ requires more work. For simplicity, here we denote
\begin{equation*}
\widehat{W}:=|Dw| + \sqrt{\lambda}|w|.
\end{equation*}
Testing \eqref{eqn-08041706} by $w$, using the ellipticity condition and H\"older's inequality, and then rearranging terms, we obtain
\begin{equation}\label{est-08051349}
\begin{split}
(\widehat{W}^2)&_{\Omega_{R/4}^{+}}\\
\lesssim& \big((|(\overline{a_{ij}}-a_{ij})D_j(\chi u)|^2)^{1/2}_{\Omega_{R/4}^{+}} + (|uD\chi|^2)^{1/2}_{\Omega_{R/4}^{+}} + (|\chi F|^2)^{1/2}_{\Omega_{R/4}} + (|\bI_{Ex\in\Omega_{R/4}^{-}}U|^2)^{1/2}_{\Omega_{R/4}}\big) (\widehat{W}^2)^{1/2}_{\Omega_{R/4}^{+}}\\
&+ \big((|\bI_{\supp(D\chi)}U|^2)^{1/2}_{\Omega_{R/4}^{+}} + (|\bI_{\supp(D\chi)}F|^2)^{1/2}_{\Omega_{R/4}^{+}}\big)(|wD\chi|^2)^{1/2}_{\Omega_{R/4}^{+}}\\
&+ \big((|\bI_{Ex\in\Omega_{R/4}^{-}}U|^2)^{1/2}_{\Omega_{R/4}^{+}} + (|\bI_{Ex\in\Omega_{R/4}^{-}}F|^2)^{1/2}_{\Omega_{R/4}^{+}}\big)
(|wD\widetilde{\chi}|^2)^{1/2}_{\Omega_{R/4}^{+}}.
\end{split}
\end{equation}
Applying the decomposition argument in the proof of \eqref{est-08041538-2} and the Poincar\'e inequality in Lemma \ref{lem-poincare} again, we have
\begin{equation}\label{est-08062324}
(|wD\chi|^2)^{1/2}_{\Omega_{R/4}^{+}} + (|wD\widetilde{\chi}|^2)^{1/2}_{\Omega_{R/4}^{+}} \lesssim (|Dw|^2)^{1/2}_{\Omega_{R/4}^{+}}.
\end{equation}
Notice that here we do not need to increase the domain of integration since we can take zero extension of $w$ outside $\Omega^{+}_{R/4}$, which is still a $W^{1,2}$ function on the upper half space.

Substituting \eqref{est-08062324} back into \eqref{est-08051349} and then applying H\"older's inequality, \eqref{est-08041538-2}, and Proposition \ref{prop-reverse-holder}, we obtain
\begin{align}
(|\widehat{W}&|^2)^{1/2}_{\Omega_{R/4}^{+}}\nonumber\\
\lesssim& (|(\overline{a_{ij}}-a_{ij})D_j(\chi u)|^2)^{1/2}_{\Omega_{R/4}^{+}} + (|uD\chi|^2)^{1/2}_{\Omega_{R/4}^{+}} + (|\chi F|^2)^{1/2}_{\Omega_{R/4}} + (|\bI_{Ex\in\Omega_{R/4}^{-}}U|^2)^{1/2}_{\Omega_{R/4}} \nonumber\\
&+ (|\bI_{\supp(D\chi)}U|^2)^{1/2}_{\Omega_{R/4}^{+}} + (|\bI_{\supp(D\chi)}F|^2)^{1/2}_{\Omega_{R/4}^{+}} + (|\bI_{Ex\in\Omega_{R/4}^{-}}F|^2)^{1/2}_{\Omega_{R/4}^{+}}\nonumber\\
\lesssim& (|(\overline{a_{ij}}-a_{ij})|^{2\mu'})^{1/(2\mu')}_{\Omega_{R/4}^{+}}(|Du|^{2\mu})^{1/(2\mu)}_{\Omega_{R/4}^{+}} + (|uD\chi|^2)^{1/2}_{\Omega_{R/4}} + (F^{2\mu})^{1/(2\mu)}_{\Omega_{R/4}}\nonumber\\
&+ (|\bI_{Ex\in\Omega_{R/4}^{-}}+\bI_{\supp(D\chi)}|^{2\mu'})^{1/(2\mu')}_{\Omega_{R/4}}(U^{2\mu})^{1/(2\mu)}_{\Omega_{R/4}}\nonumber\\
\lesssim& (\theta^{1/(2\mu')}+\gamma^{1/(2\mu')})(U^2)^{1/2}_{\Omega_R(x_0)} + (F^{2\mu})^{1/(2\mu)}_{\Omega_R(x_0)}.\label{est-08051626}
\end{align}
Hence we reach \eqref{eqn-est-W} by combining \eqref{est-08051625-1}, \eqref{est-08051625-2}, and \eqref{est-08051626}:
\begin{equation*}
\begin{split}
&(W^2)^{1/2}_{\Omega_{R/32}(x_0)} \\
&\lesssim (|W\bI_{\Omega_{R/4}^{-}}|^2)^{1/2}_{\Omega_{R/32}(x_0)} + (|W\bI_{\Omega_{R/4}^{+}}|^2)^{1/2}_{\Omega_{R/32}(x_0)}\\
&\lesssim (|W\bI_{\Omega_{R/4}^{-}}|^2)^{1/2}_{\Omega_{R/32}(x_0)} + (\widehat{W}^2)^{1/2}_{\Omega_{R/4}^{+}} + (|D((1-\chi)u)|^2)^{1/2}_{\Omega_{R/32}(x_0)} + \sqrt{\lambda}(|(1-\chi)u|^2)^{1/2}_{\Omega_{R/32}(x_0)}\\
&\lesssim (\theta^{1/(2\mu')}+\gamma^{1/(2\mu')})(U^2)^{1/2}_{\Omega_R(x_0)} + (F^{2\mu})^{1/(2\mu)}_{\Omega_R(x_0)}.
\end{split}
\end{equation*}
The estimate for $V$ follows from a linearly transformed and rescaled version of Theorem \ref{thm-halfspace}, \eqref{est-08041538-2}, and \eqref{est-08051626}. Since $V\leq \widehat{W} + |D(\chi u)| + \sqrt{\lambda}|\chi u|$, for any $q\in \big(2,\frac{2(m+2)}{m+1}\big)$,
\begin{equation*}
\begin{split}
(V^q)^{1/q}_{\Omega_{R/32}(x_0)}
\leq&
(V^q)^{1/q}_{\Omega_{R/8}^{+}}
\lesssim  (V^2)^{1/2}_{\Omega_{R/4}^+}\\
\lesssim&  (\abs{D(\chi u)}^2)^{1/2}_{\Omega_{R/4}^+} + \sqrt{\lambda}(\abs{\chi u}^2)^{1/2}_{\Omega_{R/4}^+} + (\widehat{W}^2)^{1/2}_{\Omega_{R/4}^+}\\
\lesssim&  (U^2)^{1/2}_{\Omega_R(x_0)} + N(\theta^{1/(2\mu')}+\gamma^{1/(2\mu')})(U^2)^{1/2}_{\Omega_R(x_0)} + (F^{2\mu})^{1/(2\mu)}_{\Omega_R(x_0)}\\
\lesssim&  (U^2)^{1/2}_{\Omega_R(x_0)} + N(F^{2\mu})^{1/(2\mu)}_{\Omega_R(x_0)}.
\end{split}
\end{equation*}
Hence Lemma \ref{lem-decomposition} is proved.
\end{proof}
\subsection{Level set argument and proof of Proposition \ref{prop-regularity}}
In this subsection, we prove Proposition \ref{prop-regularity}, by deriving a decay estimate of the following two level sets:
\begin{align*}
&\cA(s):= \set{x\in\Omega:\cM_{\Omega}(U^2)^{1/2} >s},\\
&\cB(s):= \set{x\in\Omega: (\gamma^{1/(2\mu')}+\theta^{1/(2\mu')})^{-1}
\cM_{\Omega}(F^{2\mu})^{1/(2\mu)} + \cM_{\Omega}(U^2)^{1/2} >s }.
\end{align*}
Here $\cM_{\Omega}$ is the Hardy-Littlewood maximal operator on $\Omega$, by taking zero extension outside: for $f\in L^1_{\rm{loc}}(\Omega)$ and $x\in\Omega$,
\begin{equation*}
\cM_{\Omega}(f)(x):=\sup_{r>0}\dashint_{B_r(x)}F
\bI_{\Omega}.
\end{equation*}
By the Hardy-Littlewood  theorem, for any $f\in L^q(\Omega)$ with $q\in [1,\infty)$, we have
\begin{equation}\label{est-08051700}
\abs{\set{x\in\Omega:\cM_{\Omega}(f)(x)>s}} \leq C\frac{\norm{f}_{L^q(\Omega)}^q}{s^q},
\end{equation}
where $C=C(d,q)$.

As mentioned before, we first prove the following local property at certain small scales.
\begin{lemma}\label{lem-levelset-local}
Under the conditions of Lemma \ref{lem-decomposition}, for any $q\in[2,2(m+2)/(m+1))$, there exists a constant $C$ depending on $(d,p,q,M,\Lambda)$, such that for all $\kappa>2^{d/2}$ and $s>0$, the following holds: if for some $R<R_1, x_0\in \overline{\Omega}$,
\begin{equation}\label{eqn-08061347}
\abs{\Omega_{R/128}(x_0)\cap\cA(\kappa s)} \geq C\big(\kappa^{-q} + \kappa^{-2}(\gamma^{1/\mu'}+\theta^{1/\mu'})\big)\abs{\Omega_{R/128}(x_0)},
\end{equation}
then $\Omega_{R/128}(x_0)\subset \cB(s)$.
\end{lemma}
\begin{proof}
It suffices to prove the contrapositive of the statement with $s=1$: suppose that there exists a point $z_0\in \Omega_{R/128}(x_0), z_0\notin \cB(1)$, we aim to prove that there exists some constant $C=C(d,p,q,M,\Lambda)$, such that
\begin{equation}\label{est-08061332}
\abs{\Omega_{R/128}(x_0)\cap\cA(\kappa)} < C\big(\kappa^{-q} + \kappa^{-2}(\gamma^{1/\mu'}+\theta^{1/\mu'})\big)\abs{\Omega_{R/128}(x_0)}.
\end{equation}
We decompose
\begin{align*}
&\Omega_{R/128}(x_0)\cap\cA(\kappa)\\
&=\{y\in\Omega_{R/128}(x_0): \sup_{r\in A_y^\kappa}\,r>R/64\}\cup \{y\in\Omega_{R/128}(x_0): \sup_{r\in A_y^\kappa}\,r\leq R/64\},
\end{align*}
where
\begin{equation*}
A_y^\kappa:=\{r:(U^2)^{1/2}_{B_r(y)}>\kappa\}.
\end{equation*}
We claim $\{y\in\Omega_{R/128}(x_0): \sup_{r\in A_y^\kappa}\,r>R/64\}=\emptyset$. By contradiction, if there exists one such $y$, take any $r\in A^\kappa_y, r>R/64$. From $B_r(y)\subset B_{2r}(z_0)$ and $z_0\notin \cB(1)$, we obtain
\begin{equation*}
(U^2)^{1/2}_{B_r(y)} \leq 2^{d/2}(U^2)^{1/2}_{B_{2r}(z_0)} \leq 2^{d/2} \cM_{\Omega}(U^2)^{1/2}(z_0)\leq 2^{d/2} \leq \kappa.
\end{equation*}
This contradicts the definition of $A^{\kappa}_y$. We are left to estimate $|\{y\in\Omega_{R/128}(x_0): \sup_{r\in A_y^\kappa}\,r\leq R/64\}|$. For any $y$ in this set and $r\in A^\kappa_y$, by noting that $B_r(y)\subset B_{R/32}(z_0)$, we apply Lemma \ref{lem-decomposition} on $\Omega_{R/32}(z_0)$ to obtain the decomposition $U\leq V+W$, satisfying
\begin{equation}\label{est-08061322}
(W^2)^{1/2}_{\Omega_{R/32}(z_0)}\leq C(\gamma^{1/(2\mu')}+\theta^{1/(2\mu')}),\quad (V^q)^{1/q}_{\Omega_{R/32}(z_0)}\leq C,
\end{equation}
where the constant $C=C(d,p,q,M,\Lambda)$. Here, to reach \eqref{est-08061322} from \eqref{eqn-est-W} and \eqref{eqn-est-V}, we used the fact that $z_0\notin \cB(1)$. Now, by $r\in A^\kappa_y$ and $U\leq V+W$ on $B_{R/32}(z_0)$, we have for any point in $\{y\in\Omega_{R/128}(x_0): \sup_{r\in A_y^\kappa}\,r\leq R/64\}$,
\begin{equation*}
\kappa < (U^2)^{1/2}_{B_r(y)} \leq \cM_\Omega(|W\bI_{\Omega_{R/32}(z_0)}|^2)^{1/2}(y) + \cM_\Omega(|V\bI_{\Omega_{R/32}(z_0)}|^2)^{1/2}(y).
\end{equation*}
Hence, we get
\begin{align}
|\{y\in&\Omega_{R/128}(x_0): \sup_{r\in A_y^\kappa}\,r\leq R/64\}|\nonumber\\
\leq& |\{\cM_\Omega(|W\bI_{\Omega_{R/32}(z_0)}|^2)^{1/2}>\kappa/2\}| + |\{\cM_\Omega(|V\bI_{\Omega_{R/32}(z_0)}|^2)^{1/2}>\kappa/2\}|\nonumber\\
\leq& C\big((\kappa/2)^{-2} \norm{W}_{L^2(\Omega_{R/32}(z_0))}^2+ (\kappa/2)^{-q}\norm{V}^q_{L^q(\Omega_{R/32}(z_0))}\big)\label{est-08061327}\\
\leq& C|\Omega_{R/32}(z_0)|(\kappa^{-2}(\gamma^{1/\mu'}+\theta^{1/\mu'}) + \kappa^{-q})\label{est-08061331}\\
\leq& C\big(\kappa^{-q} + \kappa^{-2}(\gamma^{1/\mu'}+\theta^{1/\mu'})\big)\abs{\Omega_{R/128}(x_0)},\nonumber
\end{align}
where we applied \eqref{est-08051700} to obtain \eqref{est-08061327}, and used \eqref{est-08061322} to reach \eqref{est-08061331}. The constant $C=C(d,p,q,M,\Lambda)$. This proves \eqref{est-08061332} and hence the lemma.
\end{proof}

The local property in Lemma \ref{lem-levelset-local} leads to the following global estimate by using the ``crawling of the ink spots'' lemma in \cite{KS}. For any $x_0\in\cA(\kappa s)$, we shrink the ball $\Omega_{R/128}(x_0)$ from $R=R_1$ until the first time \eqref{eqn-08061347} holds. By \eqref{est-08051700}, \eqref{eqn-08061354-1}, \eqref{eqn-08061354-2}, and the Lebesgue differentiation theorem, such $R$ exists and $R\in(0,R_1)$. Now we cover $\cA(\kappa s)$ by such balls which are almost disjoint (using the Vitali covering lemma). On each ball $\Omega_{R/128}(x_0)$, by Lemma \ref{lem-levelset-local} we have $\Omega_{R/128}(x_0)\subset \cB(s)$. Since this is an almost disjoint cover of $\cA(\kappa s)$, the following corollary is proved.
\begin{corollary}\label{cor-levelset-global}
Under the same hypothesis of Lemma \ref{lem-levelset-local}, for any $q\in [2,2(m+2)/(m+1))$, there exists a constant $C$ depending on $(d,p,q,M,\Lambda)$, such that for any $\kappa>\max\set{2^{d/2},\kappa_0}$ and
\begin{align}\label{eqn-08061354-1}
s&>s_0(d,p,q,M,\Lambda,R_1,\kappa,\gamma,\theta,\norm{U}_{L^2(\Omega)})\nonumber\\
&:=\Bigg(\frac{3\norm{U}_{L^2(\Omega)}^2}{C\kappa^2(\kappa^{-q} + \kappa^{-2}(\gamma^{1/\mu'}+\theta^{1/\mu'}))|B_{R_1/128}|}\Bigg)^{1/2},
\end{align}
we have
$$
\abs{\cA(\kappa s)} \leq C\big(\kappa^{-q} + \kappa^{-2}(\gamma^{1/\mu'}+\theta^{1/\mu'})\big)\abs{\cB(s)},
$$
where $\kappa_0$ is a constant satisfying
\begin{equation}\label{eqn-08061354-2}
C\big(\kappa_0^{-q} + \kappa_0^{-2}(\gamma^{1/\mu'}+\theta^{1/\mu'})\big) < 1.
\end{equation}
\end{corollary}

Now we give the proof of Proposition \ref{prop-regularity}.
\begin{proof}[Proof of Proposition \ref{prop-regularity}]
For any $W^{1,2}$-solution $u$, in order to prove $u\in W^{1,p}$ and the estimate \eqref{est-no-lower-order}, it suffices to prove
\begin{equation}\label{est-08061542}
\lim_{S\rightarrow \infty} p\int_0^S\abs{\cA(s)}s^{p-1}\,ds \leq C\big(R_1^{d(1-p/2)}\norm{U}_{L^2(\Omega)}^p + \norm{F}_{L^p(\Omega)}^p\big).
\end{equation}
For this, we will apply Corollary \ref{cor-levelset-global} with constants $q=p/2+(m+2)/(m+1)(>p>2\mu)$ and $(\kappa,\gamma,\theta)$ to be determined later satisfying
\begin{equation*}
\gamma < 1/(32\sqrt{d+3}),\quad \theta\in(0,1),\quad \kappa>\max\{2^{d/2},\kappa_0\}.
\end{equation*}
Changing the integral variable in \eqref{est-08061542} from $s$ to $\kappa s$, applying \eqref{est-08051700} and Corollary \ref{cor-levelset-global}, we obtain
\begin{align*}
p\int_0^S&\abs{\cA(s)}s^{p-1}\,ds = p\kappa^p\int_0^{S/\kappa}\abs{\cA(\kappa s)}s^{p-1}\,ds\\
\leq& C_0\kappa^p\int_0^{s_0}\frac{\norm{U}^2_{L^2(\Omega)}}{(\kappa s)^2}s^{p-1}\,ds + C\kappa^p\big(\kappa^{-q} + \kappa^{-2}(\gamma^{1/\mu'}+\theta^{1/\mu'})\big)\int_{s_0}^{S/\kappa}\abs{\cB(s)}s^{p-1}\,ds\\
\leq& C_1(R_1^{d(1-p/2)}\norm{U}^{p}_{L^2(\Omega)} + \norm{F}^p_{L^p(\Omega)})\\
&+ C\kappa^p\big(\kappa^{-q} + \kappa^{-2}(\gamma^{1/\mu'}+\theta^{1/\mu'})\big)\int_{0}^{S/(2\kappa)}\abs{\cA(s)}s^{p-1}\,ds,
\end{align*}
where $C_0=C_0(d,p)$, $C_1=C_1(d,p,q,M,\Lambda,\kappa,\gamma,\theta)$, and $C=C(d,p,q,M,\Lambda)$. Here for the last inequality, we used the relation
\begin{equation*}
\cB(s)\subset\cA(s/2)\cup\bigset{(\gamma^{1/(2\mu')}+\theta^{1/(2\mu')})^{-1}
\cM_{\Omega}(F^{2\mu})^{1/(2\mu)} >s/2}
\end{equation*}
and the Hardy-Littlewood theorem, then we changed the integral variable from $s$ to $s/2$. Now \eqref{est-08061542} is proved if we first choose $\kappa$ large enough and then choose $\gamma$ and $\theta$ small enough to absorb the last term on the right-hand side involving $\cA(s)$, and finally pass $S$ to the infinity.
\end{proof}

\subsection{Proof of Theorem \ref{thm-wellposedness}}
We close this section by giving the proof of our first main result Theorem \ref{thm-wellposedness}. First, from Proposition \ref{prop-regularity}, we can deduce the following a priori estimate for the original equation \eqref{eqn-07261524}, without the $L^2$ norm on the right-hand side.
\begin{corollary}\label{cor-regularity}
Let $p\in (2,2(m+2)/(m+1))$, $\gamma_0,\theta_0$ be the constants from Proposition \ref{prop-regularity}, and $u\in W^{1,p}_{\cD}$ be a weak solution the equation \eqref{eqn-07261524} with $f_i,f\in L^p(\Omega)$.
Under Assumptions \ref{ass-RF} $(\gamma_0)$, \ref{ass-separation} $(\gamma_0,m)$, and \ref{ass-smallBMO} $(\theta_0)$, there exists a positive constant $\lambda_1$ depending on $(d, p,M, \Lambda, R_1, K)$ such that if $\lambda>\lambda_1$, the following estimate holds:
\begin{equation*}
\|U\|_{L^p(\Omega)}\le C\|F\|_{L^p(\Omega)},
\end{equation*}
where $C=C(d,p,M,\Lambda)$.
\end{corollary}
The proof of the above corollary is the same as \cite[Corollary~5.2]{CDL}, which we will sketch below. We first localize by considering the equation for $u\zeta$, where $\zeta\in C_0^\infty(B_\varepsilon(x_0))$ with $x_0\in \overline\Omega$ and $\varepsilon/R_1$ sufficiently small. To apply Proposition \ref{prop-regularity}, notice that $u\zeta$ still satisfies a mixed Dirichlet-conormal problem in the form of \eqref{eqn-07282306}, if we move all the lower order terms on both the equation and the conormal boundary condition to the right-hand side. Now we can absorb the $L^2$ norm in \eqref{est-no-lower-order} by applying H\"older's inequality and choosing $\epsi$ sufficiently small. Then we use the partition of unity and choose $\lambda$ sufficiently large to remove the cut-off and absorb the $L^p$ norms of $Du$ and $u$ on the right-hand side as usual.

Now we are ready to give the proof of Theorem \ref{thm-wellposedness}.
\begin{proof}[Proof of Theorem \ref{thm-wellposedness}]
We first prove (i).

{\em Case 1:} $p=2$. This is by applying the Lax-Milgram lemma directly. In this case, actually we only need to choose $\lambda>\lambda_2(d, K,\Lambda)$ for a sufficiently large $\lambda_2$, without imposing any regularity assumptions on the coefficients and domain.

{\em Case 2:} $p\in(2,2(m+2)/(m+1)$. For $\lambda>\lambda_0:=\max\{\lambda_1,\lambda_2\}$, we have both $W^{1,2}$-solvability and $W^{1,p}$-a priori estimate Corollary \ref{cor-regularity}. Hence we only need to prove the existence of $W^{1,p}$ solutions. We first prove the solvabilty for \eqref{eqn-07282306}, noting that the only problem here is $L^p\not\subset L^2$ since our domain can be unbounded. This is solved by approximation. We approximate $f$ and $f_i$ by $f^{(n)}$ and $f^{(n)}_i$ strongly in $L^p(\Omega)$ with $f^{(n)}, f^{(n)}_i\in L^p\cap L^2$. Associated to these right-hand side functions, we can find solution $u^{(n)}\in W^{1,2}_\cD(\Omega)$ . By Proposition \ref{prop-regularity}, we have $u^{(n)}\in W^{1,p}(\Omega)$. Furthermore, by Corollary \ref{cor-regularity} $\{u^{(n)}\}_n$ is a Cauchy sequence in the space $W^{1,p}(\Omega)$. Clearly the limit $u\in W^{1,p}(\Omega)$ must be a solution to \eqref{eqn-07282306}.

From this, the original problem \eqref{eqn-07261524} can be solved by applying the method of continuity to the family of operators $tL^{(0)}_\lambda + (1-t)L_\lambda$ (and the corresponding mixed boundary condition).

{\em Case 3:} $p\in (2(m+2)/(m+3),2)$. This can be obtained from Case 2 by duality. Such argument can be found in \cite[Proof~of~Theorem~2.4]{CDL}.

Now we prove (ii). We need to lower the integrability condition for $f$ and reduce the large $\lambda$ condition to the usual sign condition $L_01\leq 0$. For the first one, we solve the follow divergence form equation for $f\in L^{p_*}(\Omega)$:
\begin{equation*}
\dv \varphi=f\,\,\text{in}\,\,\Omega,\quad \text{where}\ \ \varphi=(\varphi_i)_{i=1}^d\in W^{1,p_{*}}_{\cN}(\Omega).
\end{equation*}
Such $\varphi_i$ exists due to \cite[Lemma~7.2]{CDL}, and it satisfies
\begin{equation*}
\norm{\varphi_i}_{L^p(\Omega)}\leq C(\diam(\Omega),R_1,d,p,M)\norm{f}_{L^{p_*}(\Omega)}.
\end{equation*}
Then we solve the equation with right-hand side $D_i(f_i+\varphi_i)$. Finally, we apply the Fredholm alternative to obtain the unique solvability as well as the estimate, noting that on bounded domain the weak maximum principle holds.
\end{proof}

\section{Inhomogeneous boundary data}\label{sec-ntm}

In this section, we consider the inhomogeneous boundary value problem \eqref{190527@eq2} and prove Theorem \ref{thm-inhomo}. For this, we first solve the problem in $L^{1+\epsi}$ by applying \cite[Theorem~1.1]{TOB}. Then we improve the regularity of this solution with the help of the reverse H\"older inequality in Corollary \ref{cor-10182245}.
\subsection{Notation}
Let $\Omega$ be a Lipschitz domain locally represented by $x^1>\psi_0(x')$. For $x=(x^1,x')\in\bR\times\bR^{d-1}$, we denote the lifting up and the projection maps as
$$\Psi_0(x'):=(\psi_0(x'),x'),\quad \bP_{d-1}(x^1,x'):=x'.$$
By the definition, for any $(\gamma,m)$-flat $\Gamma$, the projection $\bP_{d-1}\Gamma$ locally is also $(\gamma,m)$-flat, of co-dimension $1$.

We consider two types of neighborhoods on the boundary: surface cubes and surface balls. We say that $Q\subset \p\Omega$ is a surface cube if $Q=\Psi_0(Q')$, where $Q'$ is a cube in $\bR^{d-1}$. For surface cubes, we denote $kQ:=\Psi_0(kQ')$. We also consider the ``cylinder'' over $Q$:
$$T(Q):=\{y\in\Omega:\dist(y,Q)<r_Q,\bP(y)\in\bP(Q)\},\quad \text{where}\,\,r_Q=\diam(Q).$$
For any point $x\in\p\Omega$, we denote $$\Delta_r(x):=B_r(x)\cap\p\Omega$$ to be the surface ball and
\begin{equation*}
\delta(x):=\dist(x,\Gamma).
\end{equation*}
In this section, we need the following truncated non-tangential maximal functions:
\begin{equation}\label{eqn-191231-1658}
\vec{N}_{r}(f)(x):=\sup_{y\in\Gamma_\beta(x)\cap B_r(x)} |f(y)|,\quad \vec{N}^r(f)(x):=\sup_{y\in\Gamma_\beta(x)\cap B_r^c(x)} |f(y)|.
\end{equation}
\subsection{An $L^{1+\epsi}$-solvability}
In this part we check that our assumptions on $\Gamma$ imply those in \cite[Theorem~1.1]{TOB}. As a corollary, the following $L^{1+\epsi}$-solvability holds.
\begin{lemma}\label{lem-1115-2129}
Let $m\in \{0,1,\ldots,d-2\}$ and $\Omega\subset \bR^d$ be a bounded domain satisfying Assumption \ref{ass-small-Lip}$(M)$. Then we can find $\gamma=\gamma(d,M)$, such that if $\Gamma$ is $(\gamma,m)$-flat, the following hold.
\begin{enumerate}
\item For any $p> 1$, \eqref{190527@eq2} has at most one solution.
\item There exists some constant $q_0=q_0(d,M)>1$, such that for any $g_\cN\in L^{q_0}(\cN)$ and $g_{\cD}\in W^{1,q_0}(\cD)$, the problem \eqref{190527@eq2} has a unique solution $u$ satisfying
\begin{equation*}
\norm{\vec{N}(Du)}_{L^{q_0}(\p\Omega)} \le C \norm{g_\cN}_{L^{q_0}(\cN)} + C\norm{g_\cD}_{W^{1,q_0}(\cD)},
\end{equation*}
where $C=C(d,M,R_0,R_1,\text{diam}(\Omega))$ is a constant.
\end{enumerate}
\end{lemma}
For this, first we can simply check that if $\Gamma$ is $(\gamma,m)$-flat, then $\cD$ satisfies the corkscrew condition relative to $\p\Omega$ in \cite{TOB} with the parameter $2/(1-\gamma)$. Now fix any small $\epsi>0$, we are going to find the constant $\gamma(\epsi,d)$, such that $(\gamma,m)$-flat implies Ahlfors $(d-2+\epsi)$-regular.

For any $x\in\Gamma$ and $r\in(0,R_1]$, due to the definition of $(\gamma,m)$-flatness at radius $r$, it can be easily verified that the surface ball $\Delta_r(x)\cap \Gamma$ can be covered by $C/\gamma^{d-2}$ balls of radius $\gamma r$, with $C$ depending only on $(d,M)$. Iterating $k$ times, we get $(C/\gamma^{d-2})^k$ number of balls of radius $\gamma^kr$, the union of which covers $\Delta_r(x)\cap\Gamma$. Thus,
\begin{equation*}
\cH^{d-2+\epsi}(\Delta_r(x)\cap\Gamma)\leq \sup_k \set{C_0 (C/\gamma^{d-2})^k(\gamma^kr)^{d-2+\epsi}}=\sup_k \set{C_0(C\gamma^{\epsi})^kr^{d-2+\epsi}}.
\end{equation*}
If we take $\gamma^{\epsi}<1/C$, the above quantity is less than $C_0 r^{d-2+\epsi}$. This means that $\Gamma$ is Ahlfors $(d-2+\epsi)$-regular. Now we have verified the conditions for \cite[Theorem~1.1]{TOB}, and thus Lemma \ref{lem-1115-2129} holds.
\subsection{A reverse H\"older inequality on boundary}
On a surface cube $Q_0$ lying in a coordinate chart, consider a function $v$ satisfying
\begin{equation}\label{eqn-190603-1}
\begin{cases}
\Delta v = 0  & \text{in }\, \Omega,\\
\frac{\p v}{\p n} = 0  & \text{on }\, \cN\cap Q_0,\\
v = 0 & \text{on }\, \cD\cap Q_0,\\
\vec{N}(Dv)\in L^{q_1}(\p\Omega),
\end{cases}
\end{equation}
where $q_1\in(1,q)$ is some constant. In this subsection, we prove the following result which plays a key role in proving Theorem \ref{thm-inhomo}.
\begin{proposition}\label{prop-190603-1}
Suppose that $v$ satisfies \eqref{eqn-190603-1}. Then for any $q<(m+2)/(m+1)$, we can find sufficiently small $\gamma_1>0$, such that if $\Omega$ and $\Gamma$ satisfy the Assumptions \ref{ass-RF}$(\gamma_1)$ and \ref{ass-separation}$(\gamma_1,m)$, then for any surface cube with $8Q\subset Q_0$, we have $\vec{N}(Dv)\in L^q(Q)$ satisfying
\begin{equation}\label{est-190603-1642}
\Bigg(\dashint_{Q} \vec{N}(Dv)^q\Bigg)^{1/q} \lesssim \dashint_{8Q} \vec{N}(Dv).
\end{equation}
\end{proposition}
We start with the following lemma which relates the boundary norm and the interior norm, for harmonic functions with purely Dirichlet or conormal boundary data.
\begin{lemma}\label{lem-10182100}
Suppose that $v$ is a harmonic function in $T(2Q)$ with either $\p v/\p \vec{n}=0$ or $v=0$ on $2Q$, then we have $\vec{N}_{r_Q}(Du)\in L^2(Q)$ satisfying
\begin{equation*}
\big(\vec{N}_{r_Q}(Dv)^2\big)_Q^{1/2} \leq C (|Dv|^2)_{T(2Q)}^{1/2},
\end{equation*}
where $C=C(d,M)$ is a constant.
\end{lemma}
See \cite[Lemma~4.4, 4.8]{OB}. The proof is by localizing the global $L^2$-result in \cite{JK81}, noting that for any $x\in Q$ we have $\Gamma_\beta(x)\cap B_{r_Q}(x)\subset T(2Q)$. Using the above lemma, we can prove the following weighted estimate for harmonic functions with mixed boundary condition. Notice that $\vec{N}(Dv)\in L^{q_1}$ guarantees that $Dv$ is an almost everywhere defined $L^{q_1}$ function on $\p\Omega$.
\begin{lemma}\label{lem-191019-0034}
Suppose that $v$ satisfies \eqref{eqn-190603-1}. In any surface cube $Q$ with $2Q\subset Q_0$ and  $4Q\cap\Gamma\neq\emptyset$, for any $\epsi'$, we have
\begin{equation*}
\int_{Q}|Dv(y)|^2\delta(y)^{1-\epsi'}\,dy \lesssim \int_{T(2Q)}|Dv(z)|^2\delta(z)^{-\epsi'}\,dz.
\end{equation*}
\end{lemma}
In this statement, the case $+\infty\leq+\infty$ is allowed.
\begin{proof}
We first construct a non-intersecting decomposition $Q\setminus\Gamma=\bigcup_j Q_j$, where the surface balls $\{Q_j\}_j$ satisfy
\begin{align}
c''\delta(y)\leq r_{Q_j}\leq c'\delta(y),\,\,\forall y\in\p\Omega,\label{eqn-10182111}\\
\sum\chi_{T(2Q_j)}\leq C(d,M,c''),\label{eqn-10182112}
\end{align}
where $c'=1/(4\sqrt d)$ and $c''=1/(16\sqrt d)$. This can be done by considering the usual Whitney decomposition $\bigcup_j Q_j'$ of $\bP(4Q)\setminus\bP(\Gamma)$ on the coordinate plane $\bR^{d-1}$, then ``lifting up'' back to surface cubes by the map $\Psi_0$. Such decomposition can also be found in \cite[Lemma~4.9]{OB}. 

Now since $c'<1/4$, we have $4Q_j\cap \Gamma=\emptyset$, which means only one type of boundary condition is prescribed on $4Q_j$. Then applying Lemma \ref{lem-10182100}, we obtain
\begin{equation*}
\int_{Q_j}\vec{N}_{r_{Q_j}}(Dv)^2 \leq C r_{Q_j}^{-1}\int_{T(2Q_j)}|Dv|^2,
\end{equation*}
where $\vec{N}_{r_{Q_j}}$ is the truncated non-tangential maximal operator defined in \eqref{eqn-191231-1658}. Using \eqref{eqn-10182111}, the point-wise inequality $|Dv(y)|\leq \vec{N}_{r_{Q_j}}(Dv)(y)$ for each $y\in Q_j$, and $\delta(y)\approx\delta(z)\approx r_{Q_j}$ for each $y\in Q_j$ and $z\in T(2Q_j)$, we have
\begin{equation*}
\int_{Q_j}|Dv(y)|^2\delta(y)^{1-\epsi'} \leq C \int_{T(2Q_j)}|Dv(z)|^2 \delta(z)^{-\epsi'}.
\end{equation*}
The lemma is proved by simply summing over $j$ and using \eqref{eqn-10182112}.
\end{proof}
Now we are in the position of proving Proposition \ref{prop-190603-1}.
\begin{proof}[Proof of Proposition \ref{prop-190603-1}]
We first claim that for any surface cube $Q$ with $4Q\subset Q_0$,
\begin{equation}\label{est-191018-2350}
\Bigg(\dashint_{Q} |Dv|^q\Bigg)^{1/q} \lesssim \dashint_{4Q} \vec{N}(Dv).
\end{equation}

\textbf{Case 1:} $4Q\cap\Gamma=\emptyset$.  Noting that $q<(m+2)/(m+1)<2$, we apply H\"older's inequality, Lemma \ref{lem-10182100}, and a rescaled version of Corollary \ref{cor-10182245} with $p=2$ to obtain
\begin{equation*}
\begin{split}
\Bigg(\dashint_{Q} |Dv|^q\Bigg)^{1/q} \lesssim \Bigg(\dashint_{Q} \vec{N}_{r_{Q}}(Dv)^q\Bigg)^{1/q} &\lesssim \Bigg(\dashint_{Q} \vec{N}_{r_{Q}}(Dv)^2\Bigg)^{1/2}\lesssim  \Bigg(\dashint_{T(2Q)}|Dv|^2\Bigg)^{1/2}\\
&\lesssim \dashint_{T(4Q)}|Dv|\lesssim \dashint_{4Q} |\vec{N}(Dv)|.
\end{split}
\end{equation*}
Here in the first and the last inequality, we used the point-wise inequality in the form
\begin{equation*}
|Dv(y)|\leq |\vec{N}_{r_Q}(Dv(\psi_0(y'),y'))|\leq |\vec{N}(Dv(\psi_0(y'),y'))|,\,\, \forall y=(y^1,y')\in T(4Q).
\end{equation*}

\textbf{Case 2:} $4Q\cap\Gamma\neq\emptyset$. In this case, we estimate as follows, with $p\in(1,2(m+2)/(m+1))$ and $\epsi'$ to be chosen later:
\begin{align}
\Bigg(\dashint_Q |Dv|^q\Bigg)^{1/q}
&\leq \Bigg(\dashint_Q |\nabla v|^2 \delta^{1-\epsi'}\Bigg)^{1/2}
\Bigg(\int_Q \delta^{(\epsi'-1)/(2/q-1)}\Bigg)^{1/q-1/2}\label{eqn-191019-0040-1}\\
&\lesssim \Bigg(\dashint_{T(2Q)}|\nabla v|^2\delta^{-\epsi'}\Bigg)^{1/2} \Bigg(\dashint_Q \delta^{(\epsi'-1)/(2/q-1)}\Bigg)^{1/q-1/2}\label{eqn-191019-0040-2}\\
&\lesssim \Bigg(\dashint_{T(2Q)}|\nabla v|^p\Bigg)^{1/p} \Bigg(\dashint_{T(2Q)}\delta^{(-\epsi')\frac{1}{1-2/p}}\Bigg)^{1/2-1/p}
\Bigg(\dashint_Q \delta^{(\epsi'-1)/(2/q-1)}\Bigg)^{1/q-1/2}\label{eqn-191019-0040-3}\\
&\lesssim \dashint_{T(4Q)}|\nabla v| \lesssim \dashint_{4Q}\vec{N}(Dv).\label{est-190603-1627}
\end{align}
Here, for the inequalities \eqref{eqn-191019-0040-1} and \eqref{eqn-191019-0040-3} we applied H\"older's inequality. In \eqref{eqn-191019-0040-2}, we applied Lemma \ref{lem-191019-0034}. In order to obtain \eqref{est-190603-1627}, we used Corollary \ref{cor-10182245}. To make sure the two integrals in \eqref{eqn-191019-0040-3} involving $\delta$ cancel, we only require that they are both finite. For this, we fix
$$p=\frac{m+2}{m+1}+q\quad\text{and}\quad\epsi'=2-1/q-2/p$$
so that
$$
\frac{-\varepsilon'}{1-2/p}>-2,\quad \frac{\varepsilon'-1}{2/q-1}>-1.
$$
Now we reach \eqref{est-191018-2350}. Note that the only difference between \eqref{est-191018-2350} and \eqref{est-190603-1642} is that we only controlled $Dv$ instead of $\vec{N}(Dv)$. But for harmonic functions actually they are equivalent. For $t\in(1/2,2)$, since $v\in W^{1,q}(\p T((1+t)Q)$ and $T((1+t)Q)$ is a Lipschitz domain, we can apply the classical $L^q$-estimate of the Dirichlet regularity problem in \cite{JK81} to obtain
\begin{equation*}
\Bigg(\dashint_{Q}\vec{N}_{r_Q/2}(Dv)^q\Bigg)^{1/q}
\lesssim \Bigg(\dashint_{\p T(1+t)Q} |D_T u|^q\Bigg)^{1/q}
\lesssim \Bigg(\dashint_{\p T(1+t)Q} |D u|^q\Bigg)^{1/q},
\end{equation*}
where $D_T$ is the tangential derivative, and we also used the fact that
$$\Gamma_\beta(x)\cap B_{r_Q/2}(x)\subset T(1+t)Q,\,\,\forall x\in Q\,\,\text{and}\,\,t>1/2.$$
Now we can take the average for $t\in(1/2,2)$ and apply Corollary \ref{cor-10182245} to obtain
\begin{align}
&\Bigg(\dashint_{Q}\vec{N}_{r_Q/2}(Dv)^q\Bigg)^{1/q}
\lesssim \Bigg(\dashint_{T(2Q)} |D v|^q\Bigg)^{1/q}
+ \Bigg(\dashint_{2Q} |D v|^q\Bigg)^{1/q}\nonumber\\
&\lesssim \Bigg(\dashint_{T(4Q)} |D v|\Bigg)
+ \Bigg(\dashint_{2Q} |D v|^q\Bigg)^{1/q}
\lesssim \dashint_{4Q}\vec{N}(Dv)
+ \Bigg(\dashint_{2Q} |D v|^q\Bigg)^{1/q}.\label{est-191019-1456}
\end{align}
The estimate for $\vec{N}^{r_Q/2}(Dv)$ is easier. For $z\in \Omega$, consider the region
$$\Lambda_{\beta}(z):= \{y\in\p\Omega: z\in\Gamma_\beta(y)\}.$$
Since $\p\Omega$ is Lipschitz, for any $z\in \Gamma_\beta(x)$ with $x\in Q, |z-x|\geq r_Q/2$, we have
\begin{equation}\label{est-191019-1433}
|\Lambda_{\beta}(z)\cap 2Q|\geq Cr^{d-1} \geq C|Q|,
\end{equation}
where $C=C(d,M)$ is a constant. Hence, for any $x\in Q$ and $z\in \Gamma_\beta(X)\cap B_{r_Q/2}^c(x)$, we have
\begin{equation*}
|D v(z)|\leq \dashint_{\Lambda_{\beta}(z)\cap 2Q}\vec{N}(Dv) \lesssim \dashint_{2Q}\vec{N}(Dv),
\end{equation*}
where in the last inequality, we used \eqref{est-191019-1433}. Taking sup in $z\in \Gamma_\beta(x)\cap B_{r_Q/2}^{c}(x)$, we obtain
\begin{equation*}
\vec{N}^{r_Q/2}(Dv)(x)\lesssim \dashint_{2Q}\vec{N}(Dv),\,\,\forall x\in Q.
\end{equation*}
Now taking $L^q$-average for $x\in Q$, we have
\begin{equation}\label{est-190603-2106}
\Bigg(\dashint_Q \vec{N}^{r_Q/2}(Dv)^q\Bigg)^{1/q}
\lesssim \dashint_{2Q}\vec{N}(Dv).
\end{equation}
Combining \eqref{est-191019-1456} and \eqref{est-190603-2106}, then applying \eqref{est-191018-2350} with $Q$ replaced by $2Q$, we have
\begin{equation*}
\begin{split}
\Bigg(\dashint_{Q}\vec{N}(Dv)^q\Bigg)^{1/q}
&\leq \Bigg(\dashint_{Q}\vec{N}_{r_Q/2}(Dv)^q\Bigg)^{1/q}
+ \Bigg(\dashint_{Q}\vec{N}^{r_Q/2}(Dv)^q\Bigg)^{1/q}\\
&\lesssim \dashint_{4Q}\vec{N}(Dv) + \Bigg(\dashint_{2Q} |D v|^q\Bigg)^{1/q}
\lesssim \dashint_{8Q}\vec{N}(Dv).
\end{split}
\end{equation*}
The proposition is proved.
\end{proof}
\subsection{Proof of Theorem \ref{thm-inhomo}}
We first state an interpolation result that will be useful in our proof. It can be simply deduced from \cite[Theorem~3.2]{Shen}. See also Remark 3.3 in the same paper.
\begin{theorem}[\cite{Shen}]\label{thm-shen}
Let $Q_0$ be a surface cube and $F\in L^1(Q_0)$. Let $p_1>1$ and $f\in L^{p_2}(Q_0)$ for some $1<p_2<p_1$. Suppose that for each (dyadic) surface cube $Q\subset \frac{1}{4}Q_0$, there exist two integrable functions $F_Q$ and $R_Q$ on $Q$ such that $|F|\leq |F_Q|+|R_Q|$ on $Q$, and
\begin{equation*}
\Bigg(\dashint_{Q}|R_Q|^{p_1}\Bigg)^{1/p_1}
\leq C_1\Bigg(\dashint_{16Q}|F| + \sup_{Q'\supset Q}\dashint_{Q'}|f|\Bigg),
\end{equation*}
\begin{equation*}
\dashint_{Q}|F_Q|\leq C_2 \sup_{Q'\supset Q}\dashint_{Q'}|f|.
\end{equation*}
Then $F\in L^{p_2}(1/4Q_0)$ and
\begin{equation*}
\Bigg(\dashint_{1/4Q_0}|F|^{p_2}\Bigg)^{1/p_2}
\leq C\Bigg(\dashint_{Q_0}|F| + (\dashint_{Q_0}|f|^{p_2})^{1/p_2}\Bigg),
\end{equation*}
where $C=C(d,M,p_1,p_2,C_1,C_2)$.
\end{theorem}
Essentially this is the boundary version of the argument used by us earlier in Section \ref{sec-W1p}.

Before giving the proof of Theorem \ref{thm-inhomo}, we first make some reduction. Since $\cD$ is a $W^{1,q}$-extension domain, we can extend $g_\cD\in W^{1,q}(\cD)$ to $\overline{g}_{\cD}\in W^{1,q}(\p\Omega)$. Also, we can extend $g_\cN\in L^q(\cN)$ by zero to $\overline{g}_{\cN}\in L^q(\p\Omega)$. According to \cite[Corollary~2.12]{Kenig-book}, the following $L^q$ Dirichlet regularity problem has a unique solution for any $1<q<2+\epsi$:
\begin{equation*}
\begin{cases}
-\Delta \overline{u} = 0  & \text{in }\, \Omega,\\
\overline{u} = \overline{g}_\cD & \text{on }\, \p\Omega,\\
\vec{N} (D\overline{u}) \in L^q(\p\Omega),
\end{cases}
\end{equation*}
with
\begin{equation*}
\norm{\vec{N}(D\overline{u})}_{L^q(\p\Omega)}\lesssim \norm{\overline{g}_\cD}_{W^{1,q}(\p\Omega)} \lesssim \norm{g_\cD}_{W^{1,q}(\cD)}.
\end{equation*}
Hence, without loss of generality, we can assume $g_\cD=0$ and $g_\cN\in L^q(\p\Omega)$.
\begin{proof}[Proof of Theorem \ref{thm-inhomo}]
From Lemma \ref{lem-1115-2129}, we can find some $q_0$ slightly larger than $1$ and a solution to the mixed problem with $g_\cN\in L^q\subset L^{q_0}$ and $g_\cD=0$, satisfying
\begin{equation}\label{est-191020-0018}
\norm{\vec{N}(Du)}_{L^{q_0}(\p\Omega)}\lesssim \norm{g_\cN}_{L^{q_0}(\cN)}.
\end{equation}
We are left to show that this solution satisfies $\vec{N}(Du)\in L^q$ for any $q\in(1,(m+2)/(m+1))$ and
\begin{equation*}
\norm{\vec{N}(Du)}_{L^q(\p\Omega)} \lesssim \norm{g_\cN}_{L^q(\cN)},
\end{equation*}
if we choose the parameters $\gamma_1$ in the Assumptions \ref{ass-RF} and \ref{ass-separation} sufficiently small.

For any surface cubes $Q_0$ and $Q$ with $16Q\subset Q_0$, we take a cutoff function $\eta \in C^\infty_c(16Q)$ with $\eta=1$ in $8Q$, and solve the following equation
\begin{equation*}
\begin{cases}
\Delta w = 0  & \text{in }\, \Omega,\\
\frac{\p w}{\p n} = \eta g_\cN  & \text{on }\, \cN,\\
w = 0 & \text{on }\, \cD,\\
\vec{N}(Dw) \in L^{q_0}(\p\Omega).
\end{cases}
\end{equation*}
Again by Lemma \ref{lem-1115-2129} such $w$ exists and satisfies the following estimate
\begin{equation*}
\Bigg(\dashint_{8Q} \vec{N}(Dw)^{q_0}\Bigg)^{1/q_0}
\lesssim \Bigg(\dashint_{16Q} |g_\cN|^{q_0}\Bigg)^{1/q_0}.
\end{equation*}
Then $v:=u-w$ satisfies \eqref{eqn-190603-1} with $Q_0$ replaced by $8Q$. Hence from Proposition \ref{prop-190603-1} with $q$ replaced by some $q+\epsi\in(q,(m+2)/(m+1))$, we have
\begin{equation*}
\begin{split}
\Bigg(\dashint_{Q} \vec{N}(Dv)^{q+\epsi}\Bigg)^{1/(q+\epsi)}
&\lesssim \dashint_{8Q} \vec{N}(Dv) \lesssim \dashint_{8Q} \vec{N}(Du) + \dashint_{8Q} \vec{N}(Dw)\\
&\lesssim \dashint_{8Q} \vec{N}(Du) + \Bigg(\dashint_{8Q} \vec{N}(Dw)^{q_0}\Bigg)^{1/q_0}
\lesssim \dashint_{8Q} \vec{N}(Du) + \Bigg(\dashint_{16Q} |g_\cN|^{q_0}\Bigg)^{1/q_0}.
\end{split}
\end{equation*}
Now we apply Theorem \ref{thm-shen} with
$$F=\vec{N}(Du)^{q_0},\quad F_Q=2^{q_0-1}\vec{N}(Dw)^{q_0},\quad R_Q=2^{q_0-1}\vec{N}(Dv)^{q_0},$$
$$f=|g_\cN|^{q_0},\quad p_1=(q+\epsi)/q_0,\quad p_2=q/q_0
$$
to obtain $\vec{N}(Du)\in L^q(Q_0/4)$ with
\begin{equation*}
\Bigg(\dashint_{Q_0/4}\vec{N}(Du)^q\Bigg)^{1/q}
\lesssim \Bigg(\dashint_{Q_0}\vec{N}(Du)^{q_0}\Bigg)^{1/q_0}
+ \Bigg(\dashint_{Q_0}|g_\cN|^q\Bigg)^{1/q}.
\end{equation*}
Since the choice of the surface cube $Q_0$ is arbitrary, the theorem is proved by using a covering argument, \eqref{est-191020-0018}, and H\"older's inequality.
\end{proof}


\def\cprime{$'$}

\end{document}